\documentclass[10pt,psamsfonts]{amsart}

\usepackage{amsmath,amssymb}
\usepackage{graphicx}
\usepackage[dvips]{psfrag}

\newtheorem{theorem}{Theorem}
\newtheorem{proposition}[theorem]{Proposition}
\newtheorem{corollary}[theorem]{Corollary}
\newtheorem{lemma}[theorem]{Lemma}
\newtheorem {remark}[theorem]{Remark}

\title[Phase portraits for a class of polynomial vector fields on
$\mathbb S^{2}$] {Phase portraits for quadratic homogeneous
polynomial vector fields on $\mathbb S^{2}$}
\author[Jaume Llibre and Claudio Pessoa]{}

\thanks{The authors are partially supported by a MCYT
grant MTM2005--06098--C02--01 and by a CIRIT  grant number
2005SGR00550.}


\begin{document}
 \maketitle

\centerline{\scshape  Jaume Llibre and Claudio Pessoa}
\medskip

{\footnotesize \centerline{ Departament de Matem\`{a}tiques,
Universitat Aut\`{o}noma de Barcelona} \centerline{08193
Bellaterra, Barcelona, Spain }
\centerline{\email{jllibre@mat.uab.es, pessoa@mat.uab.es}}}

\medskip

\bigskip

\begin{quote}{\normalfont\fontsize{8}{10}\selectfont
{\bfseries Abstract.} Let $X$ be a homogeneous polynomial vector
field of degree $2$ on $\mathbb S^2$. We show that if $X$ has at
least a non--hyperbolic singularity, then it has no limit cycles.
We give necessary and sufficient conditions for determining if a
singularity of $X$ on $\mathbb S^2$ is a center and we
characterize the global phase portrait of $X$ modulo limit cycles.
We also study the Hopf bifurcation of $X$ and we reduce the
$16^{th}$ Hilbert's problem restricted to this class of polynomial
vector fields to the study of two particular families. Moreover,
we present two criteria for studying the nonexistence of periodic
orbits for homogeneous polynomial vector fields on $\mathbb S^2$
of degree $n$.
\par}
\end{quote}

\section{Introduction and statement of the main results}

\noindent A {\it polynomial vector field $X$} in $\mathbb R^{3}$
is a vector field of the form
\begin{equation}
\label{eq:01} X=P(x,y,z)\frac{\partial}{\partial
x}+Q(x,y,z)\frac{\partial}{\partial
y}+R(x,y,z)\frac{\partial}{\partial z},
\end{equation}
where $P$, $Q$ and $R$ are polynomials in the variables $x$, $y$
and $z$ with real coefficients. We denote $n=\max\{\deg P, \deg Q,
\deg R\}$ the {\it degree} of the polynomial vector field $X$. In
what follows $X$ will denote the above polynomial vector field.

Let $\mathbb S^{2}$ be the $2$--dimensional sphere $\{(x,y,z)\in
\mathbb R^{3}\; : x^2+y^2+z^2=1\}$. A {\it polynomial vector field
$X$ on $\mathbb S^{2}$} is a polynomial vector field in $\mathbb
R^{3}$ such that restricted to the sphere $\mathbb S^{2}$ defines
a vector field on $\mathbb S^{2}$; i.e. it must satisfy the
equality
\begin{equation}
\label{eq:02} xP(x,y,z)+yQ(x,y,z)+zR(x,y,z)=0,
\end{equation}
for all points $(x,y,z)$ of the sphere $\mathbb S^{2}$.

Let $f\in \mathbb R [x,y,z]$, where $\mathbb R [x,y,z]$ denotes
the ring of all polynomials in the variables $x$, $y$ and $z$ with
real coefficients. The algebraic surface $f=0$ is an {\it
invariant algebraic surface} of the polynomial vector field $X$ if
for some polynomial $K\in \mathbb R [x,y,z]$ we have
$\displaystyle Xf=P\frac{\partial f}{\partial x} +Q\frac{\partial
f}{\partial y}+R\frac{\partial f}{\partial z}=Kf$. The polynomial
$K$ is called the {\it cofactor} of the invariant algebraic
surface $f=0$. We note that since the polynomial system has degree
$n$, then any cofactor has at most degree $n-1$.

The algebraic surface $f=0$ defines an {\it invariant algebraic
curve } $\{f=0\}\cap \mathbb S^2$ of the polynomial vector field
$X$ on the sphere $\mathbb S^2$ if
\begin{itemize}
\item[(i)] for some polynomial $K\in \mathbb R [x,y,z]$ we have
$\displaystyle Xf=P\frac{\partial f}{\partial x} +Q\frac{\partial
f}{\partial y}+R\frac{\partial f}{\partial z}=Kf$, on all the
points $(x,y,z)$ of the sphere $\mathbb S^2$, and \item[(ii)]the
intersection of the two surfaces $f=0$ and $\mathbb S^2$ is
transversal; i.e. for all points $(x,y,z)\in \{f=0\}\cap \mathbb
S^2$ we have that $\displaystyle (x,y,z)\land \left(\frac{\partial
f}{\partial x},\frac{\partial f}{\partial y},\frac{\partial
f}{\partial z}\right)\neq 0$, where $\land$ denotes the vector
cross product in $\mathbb R^3$.
\end{itemize}
Again the polynomial $K$ is called the {\it cofactor} of the
invariant algebraic curve $\{f=0\}\cap \mathbb S^2$.

Note that, if a curve $\{f=0\}\cap \mathbb S^2$ satisfies the
above definition, then it is formed by trajectories of the vector
field $X$. This justifies to call $\{f=0\}\cap \mathbb S^2$ an
invariant algebraic curve, since in this case it is invariant
under flow defined by $X$ on $\mathbb S^2$.

The equation of a plane in $\mathbb R^3$ is given for
$ax+by+cz+d=0$. Any circle on the sphere lies in a plane
$ax+by+cz+d=0$, where we can assume that $a^2+b^2+c^2=1$ and
$0\leq -d<1$. If the invariant algebraic curve $\{f=0\}\cap
\mathbb S^2$ is contained in some plane, then we say that
$\{f=0\}\cap \mathbb S^2$ is an {\it invariant circle} of the
polynomial vector field $X$ on the sphere $\mathbb S^2$. Moreover,
if the plane contains the origin, then $\{f=0\}\cap \mathbb S^2$
is an {\it invariant great circle}.

Let $U$ be an open subset of $\mathbb R^{3}$. Here a nonconstant
analytic function $H:U\rightarrow \mathbb R$ is called a {\it
first integral} of the system on $U$ if it is constant on all
solutions curves $(x(t), y(t),z(t))$ of the vector field $X$ on
$U$; i.e. $H(x(t), y(t),z(t))=$ constant for all values of $t$ for
which the solution $(x(t), y(t),z(t))$ is defined in $U$. Clearly
$H$ is a first integral of the vector field $X$ on $U$ if and only
if $XH\equiv 0$ on $U$. If $X$ is a vector field on $\mathbb S^2$,
the definition of {\it first integral} on $\mathbb S^2$ is the
same substituting $U$ by $U\cap \mathbb S^2$.

In what follows we say that two phase portraits of the vector
fields $X_{1}$ and $X_{2}$ on $\mathbb S^2$ are
\index{topologically equivalent}({\it topologically }) {\it
equivalent}, if there exists a homeomorphism $h:\mathbb S^2
\rightarrow \mathbb S^2$ such that $h$ applies orbits of $X_{1}$
into orbits of $X_{2}$, preserving or reversing the orientation of
all orbits. Similarly we define ({\it topological}) {\it
equivalence} in the Poincar\'e disc, see Section \ref{sec:05} for
a definition.

In \cite{ref:15} we studied homogeneous polynomial vector fields
of degree $2$ on $\mathbb S^{2}$, and we determined the maximum
number of invariant circles when it has finitely many invariant
circles. Moreover, we characterized the global phase portraits of
these vector fields when they have finitely many invariant
circles. Camacho \cite{ref:2} in $1981$ proved some properties of
this kind of vector fields, see also \cite{ref:22}. One of the
results that can be found in \cite{ref:15} and that will be
necessary in this paper is given below.

\begin{proposition}
\label{pro:05} Let $X$ be a homogeneous polynomial vector field of
degree $2$ on $\mathbb R^3$. Then $X$ is a polynomial vector field
on $\mathbb S^{2}$ if and only if the system associated to it can
be written as
\begin{equation}
\label{eq:04}
\begin{array}{lcl} \dot{x} & = & P(x,y,z)=
a_{1}xy+a_{2}y^2+a_{3}z^2+a_{4}xz+a_{5}yz, \\ \dot{y} & = &
Q(x,y,z)= -a_{1}x^2-a_{2}xy+a_{6}z^2+a_{7}xz+a_{8}yz, \\ \dot{z} &
= & R(x,y,z)= -a_{4}x^2-a_{8}y^2-(a_{5}+a_{7})xy-a_{3}xz-a_{6}yz.
\end {array}
\end{equation}
\end{proposition}

The main results of this paper are the following ones.

The next theorem characterizes the centers of the quadratic
homogeneous polynomial vector fields on $\mathbb S^2$.

\begin{theorem}
\label{the:08} Let $X$ be a homogeneous polynomial vector field on
$\mathbb S^2$ of degree $2$ and let $p=(0,0,-1)$ be a singularity
of $X$; i.e. the system associated to $X$ can be written in the
form \eqref{eq:04} with $a_3=a_6=0$. Then $p$ is a center of $X$
if and only if $a_8=-a_4$, $a_4^2+a_7a_5<0$ and
$a_4(a_1^2-a_2^2)+a_1a_2(a_5+a_7)=0$.
\end{theorem}

Theorem \ref{the:08} will be proved in Section \ref{sec:10}.

In the next two results we provide sufficient conditions for the
non--existence of periodic orbits (and in particular of limit
cycles) for homogeneous polynomial vector fields on $\mathbb S^2$
of arbitrary degree. For more details on limit cycles see
\cite{MHNM}.

We use the notation $\tilde{F}(u,v)=F(2u,2v,u^2+v^2-1)$, where
$F\in \mathbb R[x,y,z]$.

\begin{theorem}
\label{the:11} Let $X=(P,Q,R)$ be a homogeneous polynomial vector
field on $\mathbb S^2$. Then $X$ has no periodic orbits if at
least one of following conditions is satisfied.
\begin{itemize}
\item[(a)] The polynomial $\tilde{R}$ does not change sign on
$\mathbb R^2$ and $X$ does not have a periodic orbit passing
through the point $(0,0,1)$. \item[(b)] The function $\langle
(\tilde{P},\tilde{Q}),\nabla \tilde{R}\rangle\mid_{\tilde{R}=0}$
does not change sign and $X$ does not have a periodic orbit
passing through the point $(0,0,1)$.
\end{itemize}
\end{theorem}

Now, we use the notation $\mathcal K=aP+bQ+cR$,
$\tilde{F}_1(u,v)=F(u+a,v+b,c-(au+bv)/c)$,
$\tilde{F}_2(u,v)=F(a-(bu+cv)/a,u+b,v+c)$ and
$\tilde{F}_3(u,v)=F(u+a,b-(au+cv)/b,v+c)$ and $F\in \mathbb
R[x,y,z]$.
\begin{theorem}
\label{the:13} Suppose that the homogeneous polynomial vector
field $X=(P,Q,R)$ does not have periodic orbits on $\mathbb S^2$
intersecting the great circle $C$ determined by plane $ax+by+cz=0$
with $a^2+b^2+c^2=1$. Then it has no periodic orbits if at least
one of following conditions is satisfied.
\begin{itemize}
\item[(a)] The polynomial $\tilde{\mathcal K}_1$ does not change
sign on $\mathbb R^2$. \item[(b)] The function $\langle
(\tilde{P}_1,\tilde{Q}_1),\nabla \tilde{\mathcal
K}_1\rangle\mid_{\tilde{\mathcal K}_1=0}$ does not change sign.
\item[(c)] The polynomial $\tilde{\mathcal K}_2$ does not change
sign on $\mathbb R^2$. \item[(d)] The function $\langle
(\tilde{Q}_2,\tilde{R}_2),\nabla \tilde{\mathcal
K}_2\rangle\mid_{\tilde{\mathcal K}_2=0}$ does not change sign.
\item[(e)] The polynomial $\tilde{\mathcal K}_3$ does not change
sign on $\mathbb R^2$. \item[(f)] The function $\langle
(\tilde{P}_3,\tilde{R}_3),\nabla \tilde{\mathcal
K}_3\rangle\mid_{\tilde{\mathcal K}_3=0}$ does not change sign.
\end{itemize}
\end{theorem}

Theorems \ref{the:11} and \ref{the:13} will be proved in Section
\ref{sec:09}.

The following theorem characterizes the phase portraits of
quadratic homogeneous polynomial vector fields on $\mathbb S^2$
having at least one non--hyperbolic singularity.

\begin{theorem}
\label{the:09} Let $X$ be a homogeneous polynomial vector field on
$\mathbb S^2$ of degree $2$.
\begin{itemize}
\item[(a)] If $X$ has a linearly zero singularity, then its phase
portrait is equivalent to the phase portrait of Figure
\ref{figura3}. \item[(b)] If $X$ has a nilpotent singularity, then
its phase portrait in the Poincar\'e disc is equivalent to one of
the phase portraits of Figures \ref{figura31} or \ref{figura32}.
\item[(c)] If $X$ has a center, then its phase portrait in the
Poincar\'e disc is equivalent to one of the phase portraits of
Figures \ref{figura33} or \ref{figura35}. \item[(d)] If $X$ has a
semi--hyperbolic singularity, then its phase portrait in the
Poincar\'e disc is equivalent to one of the phase portraits of
Figures \ref{figura36} or \ref{figura37}.
\end{itemize}
\end{theorem}

We note that the unique non--hyperbolic singularity which does not
appear explicitly in the statement of Theorem \ref{the:09} is a
weak focus. Such singularities can appear in the quadratic
homogeneous polynomial vector fields on $\mathbb S^2$ and such
vector fields are studied in Theorems \ref{the:12} and
\ref{the:14}. Hence, from Theorem \ref{the:09} it follows
immediately.

\begin{corollary}
\label{cor:02} Let $X$ be a homogeneous polynomial vector field on
$\mathbb S^2$ of degree $2$. If $X$ has at least a non--hyperbolic
singularity on $\mathbb S^2$ different from a weak focus, then $X$
has no limit cycles on $\mathbb S^2$.
\end{corollary}

Theorem \ref{the:09} and Corollary \ref{cor:02} will be proved in
Section \ref{sec:06}.

In the next theorem we give an upper bound for the number of
singularities of homogeneous polynomial vector fields on $\mathbb
S^2$ of arbitrary degree.

\begin{theorem}
\label{pro:12}
 Let $X=(P,Q,R)$ be a homogeneous polynomial vector field on
$\mathbb S^2$ of degree $n$. If $X$ has finitely many
singularities on $\mathbb S^2$, then $X$ has at most $2(n^2-n+1)$
singularities on $\mathbb S^2$.
\end{theorem}

Theorem \ref{pro:12} will be proved in Section \ref{sec:11}.

The following theorem characterize modulo limit cycles the phase
portraits of quadratic homogeneous polynomial vector fields on
$\mathbb S^2$ having non--degenerate singularities.

\begin{theorem}
\label{the:12} Let $X$ be a homogenous polynomial vector field on
$\mathbb S^2$ of degree $2$. If all singularities of $X$ are
non--degenerate, then its phase portrait in the Poincar\'e disc is
topologically equivalent modulo limit cycles to one of the phase
portraits giving in Figures \ref{figura41}.
\end{theorem}

Theorem \ref{the:12} will be proved in Section \ref{sec:07}.

The next result characterizes the families of quadratic
homogeneous polynomial vector fields on $\mathbb S^2$ having a
Hopf bifurcation.

\begin{theorem}
\label{the:14} Let $X$ be a family of homogeneous polynomial
vector fields on $\mathbb S^2$ of degree $2$ having a Hopf
bifurcation. Then, doing a orthogonal linear change of variables,
the system associated to $X$ can be written as
\[
\begin{array}{lcl}
\dot{x} & = & a_1xy+a_2y^2+a_5yz, \\
\dot{y} & = & -a_1x^2-a_2xy+a_7xz+a_8yz, \\
\dot{z} & = & -(a_5+a_7)xy-a_8y^2,
\end{array}
\]
with $a_1\neq 0$, $a_2(a_5+a_7)-a_1a_8\neq 0$ and $-a_5a_7<0$.
\end{theorem}

Theorem \ref{the:14} will be proved in Section \ref{sec:12}.

We have reduced the possible existence of limit cycles to two
families; more precisely, the families \eqref{eq:142}  and
\eqref{eq:144}. In Theorem \ref{the:14} we have proved that family
\eqref{eq:144} has for some values of the parameters limit cycles.
Finally in the last section we conjecture that family
\eqref{eq:142} has no limit cycles.

\section{Poincar\'e disc}
\label{sec:05}

Let $X$ be a homogeneous polynomial vector field on $\mathbb S^2$,
then the differential system associated to it is invariant with
respect to the change of variables $(x,y,z,t)\mapsto
(-x,-y,-z,-t)$ if its degree is even, or with respect to
$(x,y,z,t)\mapsto (-x,-y,-z,t)$ if its degree is odd. Thus, in
particular the phase portrait of $X$ at the northern hemisphere of
$\mathbb S^2$ is symmetric with respect to the origin of the
sphere to the phase portrait at the southern hemisphere with the
time reverse if the degree of $X$ is even, or with the same time
if the degree of $X$ is odd. We now project the northern
hemisphere of $\mathbb S^2$ orthogonally onto the plane $\Pi$
containing the equator of $\mathbb S^2$, i.e. $\mathbb S^1$. The
orbits of $X$ on the northern hemisphere of $\mathbb S^2$ are
mapped onto certain curves of the unit disc of $\Pi$. We call this
unit disc, together with the corresponding induced phase portrait,
the {\it Poincar\'e disc}.

Now we consider the homogeneous polynomial vector field
$X=(P,Q,R)$ of degree $m$. We identify $\mathbb R^2$ as the
tangent plane to the sphere $\mathbb S^2$ at the point
$p=(a,b,c)$, i.e with the plane $ax+by+cz-1=0$, where
$a^2+b^2+c^2=1$. Suppose that $c\neq 0$, then we denote the points
of $\mathbb R^2$ as $(u+a,v+b,c-(au+bv)/c)$. Let $\pi:\mathbb
R^2\rightarrow \mathbb S^2\cap\{ax+by+cz> 0\}$ be the
diffeomorphism given by $\pi
(u,v)=|c|/\sqrt{\lambda}\left(x=u+a,\; y=v+b,\;
z=c-(au+bv)/c\right)$, where $\lambda=c^2(1+u^2+v^2)+(au+bv)^2$.
That is, $\pi$ is the inverse map of the {\it central projection}
$\pi^{-1}:\mathbb S^2\cap\{ax+by+cz> 0\}\rightarrow \mathbb R^2$
defined by
\begin{equation}
\label{eq:82} \pi^{-1} (x,y,z)=\left(u=\frac{x}{ax+by+cz},\;
v=\frac{y}{ax+by+cz},\frac{z}{ax+by+cz} \right).
\end{equation}
The homogeneous polynomial system $X$ on $\mathbb S^2$ becomes,
through the central projection $\pi^{-1}$, the differential system
\[
\begin{array}{lcl}
\dot{u} & = &\displaystyle
\frac{\sqrt{\lambda}}{|c|}(\bar{P}-(u+a)(a\bar{P}+b\bar{Q}+c\bar{R})),
\\
\dot{v} & = & \displaystyle
\frac{\sqrt{\lambda}}{|c|}(\bar{Q}-(v+b)(a\bar{P}+b\bar{Q}+c\bar{R})),
\end{array}
\]
on the plane $\mathbb R^{2}$. Here $\displaystyle
\bar{F}=F\left(\pi (u,v)\right)$. Since $X$ is a homogeneous
polynomial vector field of degree $m$ we have that
\[
\begin{array}{lcl} \dot{u} & = & \displaystyle \left(\frac{|c|}{\sqrt{\lambda}}\right)^{m-1}\left(\tilde{P}-(u+a)(a\tilde{P}+b\tilde{Q}+
c\tilde{R})\right), \\
\dot{v} & = & \displaystyle
\left(\frac{|c|}{\sqrt{\lambda}}\right)^{m-1}\left(\tilde{Q}-(v+b)(a\tilde{P}+b\tilde{Q}+
c\tilde{R})\right),
\end {array}
\]
where $\tilde{F}=F(u+a,v+b,c-(au+bv)/c)$. If $t$ denotes the
independent variable in the above differential system, then this
system becomes polynomial introducing the new independent variable
$s$ through $ds=(\sqrt{\lambda}/|c|)^{1-m}dt$, i.e.
\begin{equation}
\label{eq:134}
\begin{array}{lcl} \dot{u} & = & \mathcal P(u,v)=\displaystyle \tilde{P}-(u+a)(a\tilde{P}+b\tilde{Q}+
c\tilde{R}), \\
\dot{v} & = & \mathcal Q(u,v)=\displaystyle
\tilde{Q}-(v+b)(a\tilde{P}+b\tilde{Q}+ c\tilde{R}).
\end {array}
\end{equation}

Now in the case $a\neq 0$ the central projection (\ref{eq:82})
induces a polynomial vector field on the plane $ax+by+cz-1=0$
determined by
\begin{equation}
\label{eq:155}
\begin{array}{lcl} \dot{u} & = & \mathcal P(u,v)=\displaystyle \tilde{Q}-(u+b)(a\tilde{P}+b\tilde{Q}+
c\tilde{R}), \\
\dot{v} & = & \mathcal Q(u,v)=\displaystyle
\tilde{R}-(v+c)(a\tilde{P}+b\tilde{Q}+ c\tilde{R}),
\end {array}
\end{equation}
where $\tilde{F}=F(a-(bu+cv)/a,u+b,v+c)$; and in the case $b\neq
0$ we obtain
\begin{equation}
\label{eq:156}
\begin{array}{lcl} \dot{u} & = & \mathcal P(u,v)=\displaystyle \tilde{P}-(u+a)(a\tilde{P}+b\tilde{Q}+
c\tilde{R}), \\
\dot{v} & = & \mathcal Q(u,v)=\displaystyle
\tilde{R}-(v+c)(a\tilde{P}+b\tilde{Q}+ c\tilde{R}),
\end {array}
\end{equation}
where $\tilde{F}=F(u+a,b-(au+cv)/b,v+c)$.

The proof of the next two propositions can be found in
\cite{ref:18}.
\begin{proposition}
Let $X=(P,Q,R)$ be a homogeneous polynomial vector field of degree
$m$ on $\mathbb S^2$. Then the planar vector field induced from
$X$ through the central projection \eqref{eq:82} on the plane
$ax+by+cz-1=0$ with $a^2+b^2+c^2=1$ has degree $m$ if and only if
the great circle $\{ax+by+cz=0\}\cap \mathbb S^2$ is an invariant
circle of $X$.
\end{proposition}

In general we will consider the case $(a,b,c)=(0,0,-1)$. Thus
(\ref{eq:134}) becomes
\begin{equation}
\label{eq:111}
\begin{array}{lcl} \dot{u} & = & \mathcal P(u,v)=\displaystyle P(u,v,-1)
+uR(u,v,-1), \\
\dot{v} & = & \mathcal Q(u,v)=\displaystyle  Q(u,v,-1)
+vR(u,v,-1).
\end {array}
\end{equation}
In particular if $X$ is a homogeneous polynomial  vector field on
$\mathbb S^2$ of degree $2$ associated to (\ref{eq:04}), then
(\ref{eq:111}) becomes
\begin{equation}
\label{eq:103}
\begin{array}{lcl} \dot{u} & = &
a_{3}-a_{4}u-a_{5}v+a_{3}u^2+(a_{1}+a_{6})uv+a_{2}v^2 \\ & &
-a_{4}u^3-(a_{5}+a_{7})u^2v-a_{8}uv^2, \\
\dot{v} & = &
a_{6}-a_{7}u-a_{8}v-a_{1}u^2+(a_{3}-a_{2})uv+a_{6}v^2- \\ & &
a_{4}u^2v-(a_{5}+a_{7})uv^2-a_{8}v^3.
\end {array}
\end{equation}

\begin{proposition}
\label{pro:08} Let $X$ be a homogeneous polynomial vector field on
$\mathbb S^2$ of degree $m$, and let $f\in \mathbb R [x,y,z]$ be a
homogeneous polynomial of degree $n$ such that $f(x,y,0)\not\equiv
0$. Then $f=0$ is an invariant algebraic surface of $X$ if and
only if $f(u,v,-1)=0$ is an invariant algebraic curve of the
polynomial planar vector field given by \eqref{eq:111} induced
from $X$ through the central projection \eqref{eq:82}.
\end{proposition}

The next results jointly with their respective proofs can be found
in Camacho \cite{ref:2} or in \cite{ref:18}.

\begin{proposition}
\label{pro:06} Let $X$ be the vector field associated to
\eqref{eq:04}, and let $s_{1}$, $-s_{1}\in \mathbb S^2$ be saddle
points of $X$ with a common separatrix $l$, then $l$ is contained
in a great circle passing through $s_{1}$ and $-s_{1}$.
\end{proposition}

\section{Stereographic projection}
\label{sec:08} We identify $\mathbb R^2$ as the plane $ax+by+cz=0$
with $a^2+b^2+c^2=1$. Suppose that $c\neq 0$, then we denote the
points of $\mathbb R^2$ as $(u,v,-(au+bv)/c)$. Let $\pi:\mathbb
R^2\rightarrow \mathbb S^2\setminus\{(a,b,c)\}$ be the
diffeomorphism given by $\pi
(u,v)=1/\lambda(x=a\lambda-2c^2(a-u),\; y=b\lambda-2c^2(b-v),\;
z=c\lambda-2c(c^2+au+bv))$, where
$\lambda=c^2(1+u^2+v^2)+(au+bv)^2$. That is, $\pi$ is the inverse
map of the {\it stereographic projection} $\pi^{-1}:\mathbb
S^2\setminus\{(a,b,c)\}\rightarrow \mathbb R^2$ defined by
$\pi^{-1} (x,y,z)=$
\begin{equation}
\label{eq:136} \left(u=\frac{x-a(ax+by+cz)}{1-(ax+by+cz)},\;
v=\frac{y-b(ax+by+cz)}{1-(ax+by+cz)},\frac{z-c(ax+by+cz)}{1-(ax+by+cz)}\right).
\end{equation}
Through the stereographic projection $\pi^{-1}$ the homogeneous
polynomial system $X$ on $\mathbb S^2$ becomes the differential
system
\[
\begin{array}{lcl}
\dot{u} & = & \displaystyle
\frac{\lambda}{2c^2}(\bar{P}+(u-a)(a\bar{P}+b\bar{Q}+c\bar{R})),
\\
\dot{v} & = & \displaystyle
\frac{\lambda}{2c^2}(\bar{Q}+(v-b)(a\bar{P}+b\bar{Q}+c\bar{R})),
\end{array}
\]
on the plane $\mathbb R^{2}$. Here $\displaystyle
\bar{F}=F\left(\pi (u,v)\right)$. Since $X$ is a homogeneous
polynomial vector field of degree $m$ we have that
\[
\begin{array}{lcl} \dot{u} & = & \displaystyle \frac{1}{2c^2\lambda^{m-1}}\left(\tilde{P}+(u-a)(a\tilde{P}+b\tilde{Q}+
c\tilde{R})\right), \\ \dot{v} & = & \displaystyle
\frac{1}{2c^2\lambda^{m-1}}\left(\tilde{Q}+(v-b)(a\tilde{P}+b\tilde{Q}+
c\tilde{R})\right),
\end {array}
\]
where
$\tilde{F}=F(a\lambda-2c^2(a-u),b\lambda-2c^2(b-v),c\lambda-2c(c^2+au+bv))$.
If $t$ denotes the independent variable in the above differential
system, then this system becomes polynomial introducing the new
independent variable $s$ through $ds=\lambda^{1-m}/(2c^2)dt$, i.e.
\begin{equation}
\label{eq:137}
\begin{array}{lcl} \dot{u} & = & \mathcal P(u,v)=\displaystyle \tilde{P}+(u-a)(a\tilde{P}+b\tilde{Q}+
c\tilde{R}), \\
\dot{v} & = & \mathcal Q(u,v)=\displaystyle
\tilde{Q}+(v-b)(a\tilde{P}+b\tilde{Q}+ c\tilde{R}).
\end {array}
\end{equation}
Now the dot denotes derivative with respect to the variable $s$.
In the cases $a\neq 0$ the stereographic projection (\ref{eq:136})
induces a polynomial vector field on the plane $ax+by+cz=0$
determined by
\begin{equation}
\label{eq:148}
\begin{array}{lcl} \dot{u} & = & \mathcal P(u,v)=\displaystyle \tilde{Q}+(u-b)(a\tilde{P}+b\tilde{Q}+
c\tilde{R}), \\
\dot{v} & = & \mathcal Q(u,v)=\displaystyle
\tilde{R}+(v-c)(a\tilde{P}+b\tilde{Q}+ c\tilde{R}),
\end {array}
\end{equation}
where
$\tilde{F}=F(a\lambda_1-2a(a^2+bu+cv),b\lambda_1-2a^2(b-u),c\lambda_1-2a^2(c-v))$,
with $\lambda_1=a^2(1+u^2+v^2)+(bu+cv)^2$; and in the case $b\neq
0$ we get
\begin{equation}
\label{eq:149}
\begin{array}{lcl} \dot{u} & = & \mathcal P(u,v)=\displaystyle \tilde{P}+(u-a)(a\tilde{P}+b\tilde{Q}+
c\tilde{R}), \\
\dot{v} & = & \mathcal Q(u,v)=\displaystyle
\tilde{R}+(v-c)(a\tilde{P}+b\tilde{Q}+ c\tilde{R}),
\end {array}
\end{equation}
where
$\tilde{F}=F(a\lambda_2-2b^2(a-u),b\lambda_2-2b(b^2+au+cv),c\lambda_2-2b^2(c-v))$,
with $\lambda_2=b^2(1+u^2+v^2)+(au+cv)^2$.

The planar vector field induced by the stereographic projection
\eqref{eq:136} on the plane $ax+by+cz=0$ with $a^2+b^2+c^2=1$ will
be called the {\it stereographic projection of the vector field
$X=(P,Q,R)$ at the point} $(a,b,c)\in \mathbb S^2$.

In general we will consider the case $(a,b,c)=(0,0,1)$. Thus
(\ref{eq:137}) becomes
\begin{equation}
\label{eq:139}
\begin{array}{lcl} \dot{u} & = & \displaystyle P(2u,2v,u^2+v^2-1)
+uR(2u,2v,u^2+v^2-1), \\
\dot{v} & = & \displaystyle Q(2u,2v,u^2+v^2-1)
+vR(2u,2v,u^2+v^2-1).
\end {array}
\end{equation}

The proof of the next proposition can be find in \cite{ref:17} or
in \cite{ref:18}.
\begin{proposition}
\label{pro:16} Let $X=(P,Q,R)$ be a homogeneous polynomial vector
field of degree $m$ on $\mathbb S^2$. Then the planar vector field
induced by the stereographic projection \eqref{eq:136} on the
plane $ax+by+cz=0$ with $a^2+b^2+c^2=1$ has degree $2m$. Moreover,
it has degree $2m-1$ if $(a,b,c)=(0,0,1)$ is a singularity of $X$
on $\mathbb S^2$.
\end{proposition}

\section{Classification of centers of the quadratic homogeneous polynomial vector fields on $\mathbb S^2$}
\label{sec:10}
The classification of the centers of the quadratic
homogeneous polynomial vector fields on $\mathbb S^2$ is
equivalent to the classification of the centers of system
(\ref{eq:103}) induced from (\ref{eq:04}) by the central
projection (\ref{eq:82}).
\begin{proposition}
\label{pro:09} Let $X$ be a homogeneous polynomial vector fields
on $\mathbb S^2$ of degree $2$. Then the system associated to $X$
can be written in the form \eqref{eq:04} with $a_3=a_6=0$.
\end{proposition}
\begin{proof} By the Poincar\'e--Hopf theorem (see \cite{ref:14}), a system on
$\mathbb S^2$ always have at least one singularity. Therefore, we
can suppose that (\ref{eq:04}) has a singularity at $(0,0,-1)$,
because we can do a rotation of $SO(3)$ which preserves all the
properties of $X$. Hence, without loss of generality $(0,0)$ is a
singularity of (\ref{eq:103}). This means that we can suppose that
in (\ref{eq:04}) $a_3=a_6=0$.
\end{proof}
By Proposition \ref{pro:09}, (\ref{eq:103}) becomes
\begin{equation}
\label{eq:115}
\begin{array}{lcl} \dot{u} = P(u,v) & = & -a_{4}u-a_{5}v+
a_{1}uv+a_{2}v^2
-a_{4}u^3-\\ & & (a_{5}+a_{7})u^2v-a_{8}uv^2, \\
\dot{v} = Q(u,v) & = & -a_{7}u-a_{8}v-a_{1}u^2
-a_{2}uv-a_{4}u^2v-\\ & & (a_{5}+a_{7})uv^2 -a_{8}v^3.
\end {array}
\end{equation}

Let $\widetilde{X}=(P,Q)$. Consider the linear part of
$\widetilde{X}$ on $(0,0)$, i.e.
\[
D\widetilde{X}(0,0)=\left(\begin{array}{cc}-a_4 & -a_5
\\ -a_7 & -a_8\end{array}\right),
\]
and denote by {\rm tr}$(D\widetilde{X}(0,0))=-a_4-a_8$ and $\det(
D\widetilde{X}(0,0))=a_4a_8-a_5a_7$ its trace and its determinant.
If $(0,0)$ is a center of (\ref{eq:115}), then after a linear
change of variables and a rescaling of the time variable, it can
be written in one of the following three forms:
\[
\dot{r}  = -s+P_2(r,s), \;\; \dot{s}  =  r+ Q_2(r,s).
\]
called a {\it linear type center};
\[
\dot{r} = s+P_2(r,s), \;\; \dot{s}  =  Q_2(r,s).
\]
called a {\it nilpotent center};
\[
\dot{r} = P_2(r,s), \;\; \dot{s}  =  Q_2(r,s).
\]
called a {\it linearly zero center}, where $P_{2}$, $Q_2$ are
polynomials with terms bigger or equal than $2$.

We say that $p=(0,0,-1)$ is a {\it linear type center} of $X$ if
$(0,0)$ is a linear type center of (\ref{eq:115}). In the same way
we define {\it nilpotent centers} and {\it linearly zero centers}
of $X$.

\begin{proposition}
\label{pro:11} Let $\widetilde{X}$ be the vector field associated
to system \eqref{eq:115}. If $a_4=a_5=a_7=a_8=0$, i.e.
$D\widetilde{X}(0,0)$ is the zero matrix, then $(0,0)$ is not a
center of $\widetilde{X}$. Moreover, the phase portrait of
\eqref{eq:04} is equivalent to Figure \ref{figura3}.
\end{proposition}
\begin{proof}
By hypothesis system (\ref{eq:04}) becomes
\[
\dot{x}=y(a_1x+a_2y),\;\;\dot{y}=-x(a_1x+a_2y),\;\;\dot{z}=0.
\]
Therefore, the phase portrait of (\ref{eq:04}) is equivalent to
Figure \ref{figura3}. Hence $(0,0)$ is not a center of
$\widetilde{X}$.
\end{proof}

\begin{figure}[ptb]
\begin{center}
\includegraphics[height=0.8in,width=0.8in]{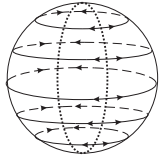}
\end{center}
\caption{The phase portrait of system (\ref{eq:04}) for
$a_3=a_4=a_{5}=a_6=a_7=a_8=0$.} \label{figura3}
\end{figure}

\begin{proposition}
\label{pro:10} The origin of \eqref{eq:115} cannot be a nilpotent
center, i.e. if \linebreak {\rm tr}$(D\widetilde{X}(0,0))=\det(
D\widetilde{X}(0,0))=0$ and $|a_4|+|a_5|+|a_7|+|a_8|\neq 0$, then
$(0,0)$ is not a center.
\end{proposition}
\begin{proof}
As {\rm tr}$(D\widetilde{X}(0,0))=\det( D\widetilde{X}(0,0))=0$,
i.e. $a_8=-a_4$ and $a_4^2+a_7a_5=0$, we suppose first that
$a_4a_7\neq 0$. Therefore, $a_5=\displaystyle -a_4^2/a_7$. Doing
the change of variables $u=r-\displaystyle s/a_4$,
$v=\displaystyle a_7s/a_4$ system (\ref{eq:115}) can be written as
\begin{equation}
\label{eq:120} \dot{r}=s+P_2(r,s), \;\;\dot{s}=Q_2(r,s).
\end{equation}
Let $\varphi(r)=\alpha_1r+\alpha_2r^2+\cdots$ the solution of
$s+P_2(r,s)=0$. Therefore, $\alpha_1=0$ and
$\alpha_2=\displaystyle (a_1a_4+a_2a_7)/a_7$. If $\alpha_2=0$,
then in the above system $(0,0)$ cannot be a center because in
this case it is not an isolated singularity. Now, if $\alpha_2\neq
0$, substituting $\varphi$ in $\psi (r)=Q_2(r,\varphi(r))$ we
obtain that $\psi(r)= \beta_2r^2+\cdots$, where
$\beta_2=\displaystyle ((a_4^2+a_7^2)(a_4a_1+a_7a_2))/(a_4a_7)$.
Hence, as $\beta_2 \neq 0$, it follows, by Theorem $67$ of page
$362$ of \cite{ref:7}, that $(0,0)$ is not a center.

Now, we have to study the case $a_4^2+a_5a_7=0$, $a_4a_7=0$ and
$|a_4|+|a_5|+|a_7|\neq 0$. Therefore, we have two possibilities
$a_4=a_5=0$, $a_7\neq 0$ and $a_4=a_7=0$, $a_5\neq 0$ for system
(\ref{eq:115}). But, as these two possibilities are equivalent by
the change of variables $(u,v)\mapsto (v,u)$, we need to study
only the last one. Therefore, after a rescaling of time given by
$d\tau=-a_5dt$ in (\ref{eq:115}), we obtain a system of the form
\begin{equation}
\label{eq:121} \dot{u}=s+P_2(u,v), \;\;\dot{v}=Q_2(u,v).
\end{equation}
Note that $a_1\neq 0$ in \eqref{eq:121}, otherwise $(0,0)$ is not
an isolated singularity. As in the previous case,  we have that
$\varphi(u)\equiv 0$ is the solution of $v+P_2(u,v)=0$ and $\psi
(u)=Q_2(u,\varphi (u))=\displaystyle (a_1u^2)/a_5+\cdots$. Hence,
by the same argument $(0,0)$ is not a center.
\end{proof}

Before proving Theorem \ref{the:08} we do one remak.
\begin{remark}
\label{ob:03} Let $U$ be an open subset of $\mathbb R^{2}$,
$F:U\rightarrow \mathbb R$ be an analytic function which is not
identically zero on $U$ and $X$ be a polynomial vector field
associated to a system of the form
\begin{equation}
\label{eq:101} \dot{x}=P(x,y), \;\;\; \dot{y}=Q(x,y),
\end{equation}
where $P$ and $Q$ are polynomials in the variables $x$ and $y$
with real coefficients. The function $F$ is an integrating factor
of this polynomial system on $U$ if one of the following three
equivalent conditions holds on $U$ $\displaystyle \frac{\partial
(FP)}{\partial x}=-\frac{\partial (FQ)}{\partial x}$, ${\rm
div}(FP,FQ)=0$, $XF=-F\; {\rm div}(P,Q)$. As usual the divergence
of the vector field $X$ is defined by ${\rm div}(X)={\rm
div}(P,Q)=\displaystyle \frac{\partial P}{\partial
x}+\frac{\partial Q}{\partial y}$.

The first integral $H$ associated to the integrating factor $F$ is
given by
\begin{equation}
\label{eq:69} H(x,y)=\int F(x,y)P(x,y)dy+h(x),
\end{equation}
where $h(x)$ is chosen in order that it satisfies $\displaystyle
\frac{\partial H }{\partial x}=-FQ$. Note that $\displaystyle
\frac{\partial H }{\partial y}=FP$, so that $XH\equiv 0$. The
function $H$ is single--valued, if $U$ is simply connected and
$F\not\equiv 0$.

Conversely, given a first integral $H$ of the system associated to
$X$ we always can find an integrating factor $F$ such that
$\displaystyle \frac{\partial H }{\partial y}=FP$ and
$\displaystyle \frac{\partial H }{\partial x}=-FQ$.

Let $V:U\rightarrow \mathbb R$ be an analytic function which is
not identically zero on $U$. The function $V$ is an inverse
integrating factor of the polynomial system (\ref{eq:101}) on $U$
if
\begin{equation}
\label{eq:102} P\frac{\partial V}{\partial x}+Q\frac{\partial
V}{\partial y}=\left(\frac{\partial P}{\partial x}+\frac{\partial
Q}{\partial y}\right)V.
\end{equation}
We note that $\{V=0\}$ is formed by orbits of system
(\ref{eq:101}) and $F= 1/V$ defines on $U\setminus \{V=0\}$ an
integrating factor of (\ref{eq:101}).
\end{remark}

\smallskip\noindent {\it Proof of Theorem \ref{the:08}.} We have
that if {\rm tr}$(D\widetilde{X}(0,0))=0$ and $\det
(D\widetilde{X}(0,0))>0$, i.e. $a_8=-a_4$ and $a_4^2+a_7a_5<0$,
then the origin can be a linear type center or a weak focus. Note
that in this case $a_5a_7\neq 0$. Hence, doing the change of
variables $u=-\displaystyle \sqrt{-(a_4^2+a_7a_5)}r/a_7+
a_4s/a_7$, $v=s$ and introducing the new time variable $\tau$
through $d\tau =\sqrt{-(a_4^2+a_7a_5)}\,\,dt$ system
(\ref{eq:115}) can be written as
\begin{equation}
\label{eq:119}
\dot{r}=-s+P_2(r,s)+P_3(r,s),\;\;\dot{s}=r+Q_2(r,s)+Q_3(r,s),
\end{equation}
where $P_2$, $P_3$, $Q_2$ and $Q_3$ are respectively the
homogeneous parts of degree $2$ and $3$ of the above system. We
have that the origin of this system is a center if and only if its
Lyapunov coefficients $V_{i}$ are zero. Therefore, we have to
determine the Lyapunov coefficients $V_{i}$ that can be obtained
from
\begin{equation}
\label{eq:116} P\frac{\partial H}{\partial r}+Q\frac{\partial
H}{\partial s}=\sum_{i=1}^{\infty}V_{i}(r^{2}+s^2)^{i+1},
\end{equation}
where $H(r,s)=\displaystyle
1/2(r^2+s^2)+\sum^{\infty}_{j=3}H_{j}(r,s)$ and $H_{j}$ are
homogeneous polynomial of degree $j$.

We will show that if $V_{1}=0$, then $V_{k}=0$ for all $k\in
\mathbb N$. From the homogeneous part of degree two, three and
four of (\ref{eq:116}), we obtain the following system
\begin{equation}
\label{eq:117}
\begin{array}{rcl}
\displaystyle -s\frac{\partial H_{2}}{\partial r}+r\frac{\partial
H_{2}}{\partial s} & \equiv & 0, \\ \displaystyle -s\frac{\partial
H_{3}}{\partial r}+r\frac{\partial H_{3}}{\partial
s}+P_{2}\frac{\partial H_{2}}{\partial r}+Q_{2}\frac{\partial
H_{2}}{\partial s} & = & 0, \\ \displaystyle -s\frac{\partial
H_{4}}{\partial r}+r\frac{\partial H_{4}}{\partial
s}+P_{2}\frac{\partial H_{3}}{\partial r}+Q_{2}\frac{\partial
H_{3}}{\partial s}+P_{3}\frac{\partial H_{2}}{\partial
r}+Q_{3}\frac{\partial H_{2}}{\partial s} & = & V_{1}(r^2+s^2)^2.
\end{array}
\end{equation}
Solving system (\ref{eq:117}), we obtain that
$V_1=((a_5-a_7)(a_4(a_1^2-a_2^2)+a_1a_2(a_5+a_7)))/(8a_7(-a_4^2-a_5a_7)^{3/2})$.
Now, as $a_4^2+a_5a_7<0$, we have that $a_5\neq a_7$. Therefore,
$V_1=0$ if and only if $a_4(a_1^2-a_2^2)+a_1a_2(a_5+a_7)=0$.
Hence, we distinguish the following four cases.

\smallskip\noindent {\it Case 1:} $a_{1}=a_{4}=0$ and $a_2\neq 0$.
In this case system (\ref{eq:115}) becomes
\begin{equation}
\label{eq:123} \dot{u} =  -a_{5}v+a_{2}v^2
-(a_{5}+a_{7})u^2v,\;\;\dot{v}  =
-a_{7}u-a_{2}uv-(a_{5}+a_{7})uv^2.
\end{equation}
This system is invariant with respect to the change of variables
$(u,v,t)\mapsto (-u,v,-t)$, i.e. it is symmetric with respect to
the straight line $u=0$. Therefore, in this case $(0,0)$ is a
center.

\smallskip\noindent {\it Case 2:} $a_{2}=0$ and $a_1\neq 0$.
Then $a_4=0$. In this case system (\ref{eq:115}) becomes
\begin{equation}
\label{eq:124}\dot{u} = -a_{5}v+a_{1}uv
-(a_{5}+a_{7})u^2v,\;\;\dot{v} =
-a_{7}u-a_{1}u^2-(a_{5}+a_{7})uv^2.
\end{equation}
We have that system $(\ref{eq:124})$ becomes $(\ref{eq:123})$ by
the change of variables $(u,v)\mapsto (v,u)$. Therefore, this case
is equivalent to Case $1$.

\smallskip\noindent {\it Case 3:} $a_{1}=a_{2}=0$. In this
case system (\ref{eq:115}) becomes
\begin{equation}
\label{eq:125}
\begin{array}{lcl} \dot{u} & = & -a_4u-a_{5}v-a_4u^3
-(a_{5}+a_{7})u^2v+a_4uv^2, \\
\dot{v} & = & -a_{7}u+a_{4}v-a_4u^2v-(a_{5}+a_{7})uv^2+a_4v^3.
\end {array}
\end{equation}
This system has a first integral given by
$H(u,v)=((a_7+a_5)u^2-2a_4uv+a_5)/((a_5-a_7)u^2+2a_4uv+2a_5v^2+a_5)$.
As $H$ is defined in $(0,0)$, it follows that in this case $(0,0)$
is a center.

\smallskip\noindent {\it Case 4:} $a_{1}a_2\neq 0$. In this
case $V_1=0$ implies that $a_5=\displaystyle
(a_4(a_2^2-a_1^2)-a_1a_2a_7)/(a_1a_2)$. Note that
$-(a_4^2+a_5a_7)>0$ becomes $\displaystyle
((a_1a_4+a_2a_7)(a_1a_7-a_2a_4))/(a_1a_2)>0$. Now, this case is
equivalent to Case $1$ doing the orthogonal linear change of
variables
\[
\left(\begin{array}{c}x
\\ y \\ z \end{array}\right)=\left(\begin{array}{ccc}\displaystyle
-\frac{a_2}{\sqrt{a_1^2+a_2^2}} & \displaystyle
\frac{a_1}{\sqrt{a_1^2+a_2^2}} & 0 \\
\displaystyle \frac{a_1}{\sqrt{a_1^2+a_2^2}} & \displaystyle
\frac{a_2}{\sqrt{a_1^2+a_2^2}} & 0 \\ 0 & 0 & 1
\end{array}\right)\left(\begin{array}{c} \tilde{x}
\\ \tilde{y} \\ \tilde{z} \end{array}\right),
\]
where the $a_i$'s denote the coefficients of $X$ in the variables
$(x,y,z)$ satisfying the hypothesis. In fact, denoting by
$\tilde{a}_i$ the coefficients of system \eqref{eq:04} in the
variables $(\tilde{x},\tilde{y},\tilde{z})$, we have that
$\tilde{a}_1=\tilde{a}_3= \tilde{a}_4=\tilde{a}_6=\tilde{a}_8=0$,
$\tilde{a}_2=-\sqrt{a_2^2+a_1^2}$,
$\tilde{a}_5=(a_1a_7-a_2a_4)/a_1$,
$\tilde{a}_7=-(a_1a_4+a_2a_7)/a_2$. Note that $\tilde{a}_2\neq 0$
and $\tilde{a}_5\tilde{a}_7<0$. This completes the proof of the
theorem. \hfill $\Box$

\section{Two criteria for determining the nonexistence of limit cycles for homogeneous polynomial vector fields on $\mathbb S^2$}
\label{sec:09} Let $X=(P,Q,R)$ be a homogeneous polynomial vector
field on $\mathbb S^2$ of degree $m$. The next proposition will be
necessary for proving the results of this section, see its proof
in \cite{ref:15}.
\begin{proposition}
\label{pro:02} Let $f(x,y,z)=x^2+y^2+z^2-1$ and let $X$ be a
homogeneous polynomial vector field on $\mathbb S^2$. Then $f$ is
a first integral of $X$ in $\mathbb R^3$, i. e.
$P(x,y,z)x+Q(x,y,z)y+R(x,y,z)z=0$ for all $(x,y,z)\in \mathbb
R^3$.
\end{proposition}

Let $(a,b,c)\in \mathbb S^2$, we can suppose that $X$ does not
have a limit cycle such that $(a,b,c)$ belongs it, because we can
do a rotation of $SO(3)$ which preserves all properties of $X$.
Hence, the study of the existence of limit cycles of $X$ on
$\mathbb S^2$ is equivalent to study the existence of limit cycles
of the planar vector field induced by $X$ through the
stereographic projection (\ref{eq:136}) at the point $(a,b,c)$. We
will consider the cases when $(a,b,c)$ is the point $(0,0,1)$,
i.e. the planar vector field determined by system (\ref{eq:139}).
We denote this planar vector field by
\[
\bar{X} =
(\bar{P},\bar{Q})=(\tilde{P}+u\tilde{R},\tilde{Q}+v\tilde{R}).
\]
Here $\tilde{F}(u,v)=F(2u,2v,u^2+v^2-1)$ with $F\in \mathbb
R[x,y,z]$.

Now we state a result that can be found in \cite{ref:19} and that
will be necessary for the proof of Theorem \ref{the:11}.
\begin{theorem}
\label{the:10} Let $(\mathcal P,\mathcal Q)$ be a $C^1$ vector
field defined in an open subset $U$ of \;$\mathbb R^2$,
$(\alpha(t),\beta(t))$ a periodic solution of $(\mathcal
P,\mathcal Q)$ of period $\tau$, $K:U\longrightarrow \mathbb R$ a
$C^1$ map such that $\int^{\tau}_0K(\alpha(t),\beta(t))dt\neq 0$,
and $f=f(x,y)$ a $C^1$ solution of the linear partial differential
equation $\displaystyle\mathcal P\frac{\partial f}{\partial
x}+\mathcal Q\frac{\partial f}{\partial y}=Kf$. Then the closed
trajectory $\gamma=\{(\alpha(t),\beta(t))\in U:\; t\in [0,\tau]\}$
is contained in $\Sigma=\{(x,y)\in U:\; f(x,y)=0\}$, and $\gamma$
is not contained in a period annulus of $(\mathcal P,\mathcal Q)$.
Moreover, if the vector field $(\mathcal P,\mathcal Q)$ is
analytic, then $\gamma$ is a limit cycle.
\end{theorem}

\smallskip\noindent {\it Proof of Theorem \ref{the:11}.}
Consider the function $f(u,v)=u^2+v^2+1$. Using the previous
notation we have that
\[
\bar{P}\frac{\partial f}{\partial u}+\bar{Q}\frac{\partial
f}{\partial v}=
2u\tilde{P}+2u^2\tilde{R}+2v\tilde{Q}+2v^2\tilde{R}=2u\tilde{P}+2v\tilde{Q}+2(u^2+v^2)\tilde{R}.
\]
Note that, by Proposition \ref{pro:02},
$2u\tilde{P}+2v\tilde{Q}+(u^2+v^2-1)\tilde{R}=0$. Therefore,
\[
\bar{P}\frac{\partial f}{\partial u}+\bar{Q}\frac{\partial
f}{\partial
v}=-(u^2+v^2-1)\tilde{R}+2(u^2+v^2)\tilde{R}=\tilde{R}(u^2+v^2+1)=\tilde{R}f.
\]
Therefore, $f=0$ is an invariant algebraic curve of $\bar{X}$ and
$\tilde{R}$ is its cofactor.

Using the notation of Theorem \ref{the:10} we assume that $\gamma$
is a periodic orbit of $\bar{X}$. Now, if $\tilde{R}$ does not
change sign in $\mathbb R^2$, then $\int_\gamma\tilde{R}dt\neq 0$,
and since $\{(u,v)\in \mathbb R^2:\; f(u,v)=0\}=\emptyset$, by
Theorem \ref{the:10} $\bar{X}$ does not have periodic orbits.
Hence, $X$ does not have periodic orbits on $\mathbb S^2$. This
proves statement (a).

Now we will prove statement (b). Note that
\[
\langle \bar{X},\nabla
\tilde{R}\rangle\mid_{\tilde{R}=0}=\left(\tilde{P}\frac{\partial
\tilde{R}}{\partial u}+\tilde{Q}\frac{\partial \tilde{R}}{\partial
v}\right)\mid_{\tilde{R}=0}=\langle (\tilde{P},\tilde{Q}),\nabla
\tilde{R}\rangle\mid_{\tilde{R}=0}.
\]
Therefore, if $\langle (\tilde{P},\tilde{Q}),\nabla
\tilde{R}\rangle\mid_{\tilde{R}=0}$ does not change sign, then the
vector field $\bar{X}$ is transversal to the curve $\tilde{R}=0$,
so there are no periodic orbits of $\bar{X}$ intersecting the
curve $\tilde{R}=0$. Hence, if $X$ has a periodic orbit $\gamma$,
then $\gamma \subset \mathbb R^2\setminus \{\tilde{R}=0\}$. Thus,
$\int_\gamma\tilde{R}dt\neq 0$, and so by Theorem \ref{the:10} we
get $\gamma \subset \{(u,v)\in \mathbb R^2:\;
f(u,v)=0\}=\emptyset$. Therefore, $\bar{X}$ does not have periodic
orbits, i.e. $X$ does not have periodic orbits on $\mathbb S^2$.
\hfill $\Box$

{From} the proof of Theorem \ref{the:11} we have:

\begin{corollary}
The vector field $\bar{X}$ has the invariant algebraic curve
$f(u,v)=u^2+v^2+1=0$ with cofactor $K_f=\tilde{R}$.
\end{corollary}

We see in the beginning of this section that the study of the
existence of limit cycles of homogeneous vector fields $X$ on
$\mathbb S^2$ is equivalent to study the existence of limit cycles
of the planar vector fields induced by $X$ through the
stereographic projection at the point $(a,b,c)$. But sometimes it
is necessary to do a rotation in order that $(a,b,c)$ does not
belong to a limit cycle of $X$. In fact, if $(a,b,c)$ belongs to
the limit cycle $\gamma$ then, as the phase portrait of $X$ is
symmetric with respect to the origin, we have that $-(a,b,c)\in
-\gamma=\{-(x,y,z):\;(x,y,z)\in \gamma\}$. If degree of $X$ is
even, by Proposition $4$ of \cite{ref:17}, $\gamma \neq -\gamma$.
Therefore, if the degree of $X$ is even, it is not necessary to do
a rotation. Now if degree of $X$ is odd, by Proposition $6$ of
\cite{ref:17} we know that the case $\gamma=-\gamma$ is possible.

Now, we will consider the the planar vector fields $\bar{X}_i$,
$i=1,2,3$, determined by systems \eqref{eq:134}, \eqref{eq:155}
and \eqref{eq:156}, i.e.
\[
\begin{array}{lcl}
\bar{X}_1 & = &
(\bar{P}_1,\bar{Q}_1)=(\tilde{P}_1-(u+a)\tilde{\mathcal
K}_1,\tilde{Q}_1-(v+b)\tilde{\mathcal K}_1),
\\ \bar{X}_2 & = &
(\bar{P}_2,\bar{Q}_2)=(\tilde{Q}_2-(u+b)\tilde{\mathcal
K}_2,\tilde{R}_2-(v+c)\tilde{\mathcal K}_2),
\\ \bar{X}_3 & = & (\bar{P}_3,\bar{Q}_3)=(\tilde{P}_3-(u+a)\tilde{\mathcal K}_3,\tilde{R}_3-(v+c)\tilde{\mathcal K}_3).
\end{array}
\]
Here $\mathcal K=aP+bQ+cR$,
$\tilde{F}_1(u,v)=F(u+a,v+b,c-(au+bv)/c)$,
$\tilde{F}_2(u,v)=F(a-(bu+cv)/a,u+b,v+c)$ and
$\tilde{F}_3(u,v)=F(u+a,b-(au+cv)/b,v+c)$ and $F\in \mathbb
R[x,y,z]$.

\smallskip\noindent {\it Proof of Theorem \ref{the:13}.} First
we will prove statement $(a)$. Consider the function
\[
\begin{array}{lcl}
f_1(u,v) & = & \displaystyle (u+a)^2+(v+b)^2+\left(c-\frac{au+bv}{c}\right)^2 \\
& = & \displaystyle
\frac{a^2+c^2}{c^2}u^2+\frac{2ab}{c^2}uv+\frac{b^2+c^2}{c^2}v^2+1.
\end{array}
\]
Using the previous notation we have that
\[
\begin{array}{lcl}
\displaystyle \bar{P}_1\frac{\partial f_1}{\partial
u}+\bar{Q}_1\frac{\partial f_1}{\partial v} & = & \displaystyle
\tilde{P}_1\frac{\partial f_1}{\partial
u}+\tilde{Q}_1\frac{\partial f_1}{\partial
v}-2\tilde{\mathcal K}_1(f_1-1)-\frac{2}{c^2}\tilde{\mathcal K}_1(au+bv) \\
& = & \displaystyle -2\tilde{\mathcal
K}_1(f_1-1)+\frac{2}{c^2}(c^2u\tilde{P}_1+c^2v\tilde{Q}_1-c(au+bv)\tilde{R}_1).
\end{array}
\]
Note that, by Proposition \ref{pro:02},
$(u+a)\tilde{P}_1+(v+b)\tilde{Q}_1+(c-(au+bv)/c)\tilde{R}_1=0$,
i.e.
$c^2u\tilde{P}_1+c^2v\tilde{Q}_1-c(au+bv)\tilde{R}_1=-c^2\tilde{\mathcal
K}_1$. Therefore, $\displaystyle \bar{P}_1\frac{\partial
f_1}{\partial u}+\bar{Q}_1\frac{\partial f_1}{\partial
v}=-2\tilde{\mathcal K}_1(f_1-1)-2\tilde{\mathcal
K}_1=-2\tilde{\mathcal K}_1f_1$. Hence, $f_1=0$ is an invariant
algebraic curve of $\bar{X}_1$ and $-2\tilde{\mathcal K}_1$ is its
cofactor.

Using the notation of Theorem \ref{the:10} we assume that $\gamma$
is a periodic orbit of $\bar{X}_1$. Now, if $\tilde{\mathcal K}_1$
does not change sign in $\mathbb R^2$, then
$\int_\gamma\tilde{\mathcal K}_1dt\neq 0$, and since $\{(u,v)\in
\mathbb R^2:\; f_1(u,v)=0\}=\emptyset$, by Theorem \ref{the:10}
$\bar{X}_1$ does not have periodic orbits. Hence, $X$ does not
have periodic orbits on $\mathbb S^2$. This proves statement (a).

Now we will prove statement (b). Note that
\[
\langle \bar{X}_1,\nabla \tilde{\mathcal
K}_1\rangle\mid_{\tilde{\mathcal
K}_1=0}=\left(\tilde{P}_1\frac{\partial \tilde{\mathcal
K}_1}{\partial u}+\tilde{Q}_1\frac{\partial \tilde{\mathcal
K}_1}{\partial v}\right)\mid_{\tilde{\mathcal K}_1=0}=\langle
(\tilde{P}_1,\tilde{Q}_1),\nabla \tilde{\mathcal
K}_1\rangle\mid_{\tilde{\mathcal K}_1=0}.
\]
Therefore, if $\langle (\tilde{P}_1,\tilde{Q}_1),\nabla
\tilde{\mathcal K}_1\rangle\mid_{\tilde{\mathcal K}_1=0}$ does not
change sign, then the vector field $\bar{X}_1$ is transversal to
the curve $\tilde{\mathcal K}_1=0$, so there are no periodic
orbits of $\bar{X}_1$ intersecting the curve $\tilde{\mathcal
K}_1=0$. Hence, if $X$ has a periodic orbit $\gamma$, then $\gamma
\subset \mathbb R^2\setminus \{\tilde{\mathcal K}_1=0\}$. Thus,
$\int_\gamma\tilde{\mathcal K}_1dt\neq 0$, and so by Theorem
\ref{the:10} we get $\gamma \subset \{(u,v)\in \mathbb R^2:\;
f_1(u,v)=0\}=\emptyset$. Therefore, $\bar{X}_1$ does not have
periodic orbits, i.e. $X$ does not have periodic orbits on
$\mathbb S^2$.

The proof of the other statements is similar replacing $f_1$ by
\[
\begin{array}{lcl}
f_2(u,v) & = & \displaystyle (u+b)^2+(v+c)^2+\left(a-\frac{bu+cv}{a}\right)^2 \\
\end{array}
\]
in the statements $(c)$, $(d)$ and by
\[
\begin{array}{lcl}
f_3(u,v) & = & \displaystyle (u+a)^2+(v+c)^2+\left(b-\frac{au+cv}{b}\right)^2 \\
\end{array}
\]
in the statements $(e)$, $(f)$. \hfill $\Box$

>From the proof of Theorem \ref{the:13}, using the previous
notation, we have the following corollaries:

\begin{corollary}
The algebraic curve $f_i=0$ is an invariant algebraic curve of
vector field $\bar{X}_i$ with cofactor $K_{f_i}=-2\tilde{\mathcal
K}_i$, for $i=1,2,3$.
\end{corollary}

\begin{corollary}
In the assumptions of Theorem \ref{the:13}, if $C$ is a periodic
orbit, then $C$ is the unique periodic orbit of $X$ on $\mathbb
S^2$.
\end{corollary}

We say that an orbit $\gamma$ of a homogeneous polynomial vector
field on $\mathbb S^2$ is {\it convex} if for any $p\in \gamma$
the great circle tangent to $\gamma$ at $p$ intersects $\gamma$
only at $p$, or $\gamma$ is entirely contained in the great
circle.

\begin{proposition}
\label{pro:14} Let $X=(P,Q,R)$ be a quadratic homogeneous
polynomial vector field on $\mathbb S^2$. Any great circle in
$\mathbb S^2$ has at most four tangencies with the orbits of $X$
or is formed by orbits of $X$.
\end{proposition}

\begin{proposition}
\label{pro:17} Let $X$ be a homogeneous polynomial vector field on
$\mathbb S^2$ of degree $2$. Any periodic orbit of $X$ on $\mathbb
S^2$ is convex.
\end{proposition}

Propositions \ref{pro:14} and \ref{pro:17} also can be found in
Camacho \cite{ref:2} or \cite{ref:18}.

By Proposition \ref{pro:17} we have that in the case that $X$ has
degree $2$ the hypotheses of Theorem \ref{the:13} is not
necessary, i.e. we can suppose that all periodic orbits of $X$ are
contained on the hemisphere determined by the plane $ax+by+cz=0$
such that $(a,b,c)$ belongs it.

\section{Phase portraits for homogeneous polynomial vector fields on $\mathbb S^2$ of degree $2$ having some non--hyperbolic singularity}
\label{sec:06}

Let $p$ be a singularity of a homogeneous polynomial vector field
$X$ on $\mathbb S^2$ of degree $2$. By Proposition \ref{pro:09},
we can suppose without loss of generality that $p=(0,0,-1)$, then
the classification of singularities of $X$ is equivalent to the
classification of singularities of (\ref{eq:115}) at the origin.

Using the notation of the previous section we say that $p$ is a
{\it hyperbolic singularity} of $X$ if the eigenvalues of
$D\widetilde{X}(0,0)$ have real part distinct of zero. We say that
$p$ is a {\it non--degenerate singularity} if $\det(
D\widetilde{X}(0,0))=a_4a_8-a_5a_7\neq 0$, a {\it semi--hyperbolic
singularity} if {\rm tr}$(D\widetilde{X}(0,0))=-a_4-a_8 \neq 0$
and $\det( D\widetilde{X}(0,0))=0$, a {\it nilpotent singularity}
if {\rm tr}$(D\widetilde{X}(0,0))=\det( D\widetilde{X}(0,0))=0$
and $|a_4|+|a_5|+|a_7|+|a_8|\neq 0$, and {\it linearly zero
singularity} if $D\widetilde{X}(0,0)$ is the null matrix.

The next result will be necessary for proving the results of this
section and its proof can be found in \cite{ref:15}.

\begin{proposition}
\label{cor:01} Let $X$ be a homogeneous polynomial vector field of
degree $2$ on $\mathbb S^2$. Suppose that $X$ has at least one
invariant circle on $\mathbb S^2$, then $X$ has no limit cycles.
\end{proposition}

\smallskip\noindent {\it Proof of Theorem \ref{the:09}.}
Statement $(a)$ is exactly Proposition \ref{pro:11}. Therefore, we
have to prove only statements $(b)$, $(c)$ and $(d)$.

In what follows, by Proposition \ref{pro:09}, we can suppose that
the system associated to $X=(P,Q,R)$ is in the form (\ref{eq:04})
with $a_3=a_6=0$. Therefore, if $R\not \equiv 0$ then the equator
of $\mathbb S^2$, i.e. $\mathbb S^1$ is not invariant by $X$,
because $z$ is not a factor of $R$. Hence, in this case, on the
Poincar\'e disc associated to $\tilde{X}$ we represent $\mathbb
S^1$ by a circle formed by dots.

In the case that $X$ has a nilpotent singularity, to determine its
phase portrait, from the proof of Proposition \ref{pro:10}, it is
sufficient to study the systems
\begin{equation}
\label{eq:132}
\begin{array}{lcl} \dot{x} & = &
a_{1}xy+a_{2}y^2+a_{4}xz-\displaystyle \frac{a_{4}^2}{a_7}yz, \\
\dot{y} & = & = -a_{1}x^2-a_{2}xy+a_{7}xz-a_{4}yz, \\ \dot{z} & =
&
-a_{4}x^2+a_{4}y^2-\left(a_{7}-\displaystyle\frac{a_{4}^2}{a_7}\right)xy,
\end {array}
\end{equation}
with $a_4a_7\neq 0$, and
\begin{equation}
\label{eq:133} \dot{x} = a_{1}xy+a_{2}y^2+a_{5}yz, \;\;\dot{y} =
-a_{1}x^2-a_{2}xy, \;\; \dot{z}  =  -a_{5}xy,
\end{equation}
with $a_5\neq 0$. Now system (\ref{eq:132}) is equivalent to
system (\ref{eq:133})  by the orthogonal linear change of
variables
\[
\left(\begin{array}{c}x
\\ y \\ z \end{array}\right)=\left(\begin{array}{ccc}\displaystyle
\frac{a_4}{\sqrt{a_4^2+a_7^2}} & \displaystyle
-\frac{a_7}{\sqrt{a_4^2+a_7^2}} & 0 \\
\displaystyle \frac{a_7}{\sqrt{a_4^2+a_7^2}} & \displaystyle
\frac{a_4}{\sqrt{a_4^2+a_7^2}} & 0 \\ 0 & 0 & 1
\end{array}\right)\left(\begin{array}{c} \tilde{x}
\\ \tilde{y} \\ \tilde{z} \end{array}\right),
\]
where $a_i$ are the coefficients of system (\ref{eq:132}).
Denoting by $\tilde{a}_i$ the coefficients of system
\eqref{eq:132} in the variables $(\tilde{x},\tilde{y},\tilde{z})$,
we have that $ \tilde{a}_1=(a_1a_4+a_2a_7)/\sqrt{a_4^2+a_7^2}$,
$\tilde{a}_2=(a_2a_4-a_1a_7)/\sqrt{a_4^2+a_7^2}$,
$\tilde{a}_5=-(a_4^2+a_7^2)/a_7$. Note that $\tilde{a}_5\neq 0$.
Therefore, we have to study only system (\ref{eq:133}) or
equivalently system (\ref{eq:121}) induced from it by the central
projection. We distinguish two cases.

\smallskip\noindent {\it Case 1:} $a_{1}\neq 0$. By the proof
of Proposition \ref{pro:10} we have that $v=\varphi (u)\equiv 0$
is the solution of $v+P_2(u,v)=0$. Hence, $\psi (u)=Q_2(u,\varphi
(u))=\displaystyle a_{1}u^2/a_5$ and $\displaystyle \frac{\partial
P_2}{\partial u}(u,\varphi (u))+\displaystyle \frac{\partial
Q_2}{\partial v}(u,\varphi (u))=\frac{a_2}{a_5}u$. Then, by
Theorem $67$ of page $362$ of \cite{ref:7} it follows that $(0,0)$
is a cusp.

In the case $a_2\neq 0$ the vector field $X$ determined by system
(\ref{eq:133}) does not have singularities on $\mathbb S^1$.
Moreover, $X$ has four singularities $(0,0,\pm 1)$,
$(0,a_5/(\sqrt{a_2^2+a_5^2}),\linebreak
-a_2/(\sqrt{a_2^2+a_5^2}))$ and
$(0,-a_5/(\sqrt{a_2^2+a_5^2}),a_2/(\sqrt{a_2^2+a_5^2}))$
corresponding to the singularities $(0,0)$ and
$\left(0,a_5/a_2\right)$ of (\ref{eq:121}). We have that the
eigenvalues associated to $\left(0,a_5/a_2\right)$ are $(-a_1\pm
\sqrt{a_1^2-4(a_2^2+a_5^2)})/2a_2$. Thus, $\left(0,a_5/a_2\right)$
can be a node or a focus of (\ref{eq:121}). Now, by the orthogonal
linear change of variables
\[
\left(\begin{array}{c} x \\
y \\ z
\end{array}\right)=\left(\begin{array}{ccc} 1 & 0 & 0
\\ 0 & \displaystyle \frac{a_2}{\sqrt{a_2^2+a_5^2}} & \displaystyle -\frac{a_5}{\sqrt{a_2^2+a_5^2}} \\
0 & \displaystyle \frac{a_5}{\sqrt{a_2^2+a_5^2}} & \displaystyle
\frac{a_2}{\sqrt{a_2^2+a_5^2}} \end{array}
\right)\left(\begin{array}{c} \tilde{x} \\ \tilde{y}
\\ \tilde{z} \end{array}\right),
\]
the system associated to $X$ becomes
\begin{equation}
\label{eq:154}
\begin{array}{lcl}
\dot{\tilde{x}} & = &
P(\tilde{x},\tilde{y},\tilde{z})=\displaystyle
\frac{a_1a_2}{\sqrt{a_2^2+a_5^2}}\tilde{x}\tilde{y}-\frac{a_1a_5}{\sqrt{a_2^2+a_5^2}}\tilde{x}\tilde{z}+
a_2\tilde{y}^2-a_5\tilde{y}\tilde{z}, \\
\dot{\tilde{y}} & = & Q(\tilde{x},\tilde{y},\tilde{z})= -\displaystyle \frac{a_1a_2}{\sqrt{a_2^2+a_5^2}}\tilde{x}^2-a_2\tilde{x}\tilde{y}+a_5\tilde{x}\tilde{z}, \\
\dot{\tilde{z}} & = & R(\tilde{x},\tilde{y},\tilde{z})=
\displaystyle \frac{a_1a_5}{\sqrt{a_2^2+a_5^2}}\tilde{x}^2.
\end{array}
\end{equation}
Note that in these coordinates the singularity
$(0,a_5/(\sqrt{a_2^2+a_5^2}),-a_2/(\sqrt{a_2^2+a_5^2}))$
correspond to the point $(0,0,-1)$. Hence, using the notation of
Section \ref{sec:09}, we have that
$\tilde{R}(u,v)=R(2u,2v,u^2+v^2-1)=
4a_1a_5/\sqrt{a_2^2+a_5^2}u^2$. Thus, $\tilde{R}$ does not change
sign and by Theorem \ref{the:11}, $X$ does not have periodic
orbits on $\mathbb S^2$. So the phase portrait of $\tilde{X}$ on
the Poincar\'e disc is equivalent to one of Figure \ref{figura31}.

\begin{figure}[ptb]
\begin{center}
\includegraphics[height=0.8in,width=2in]{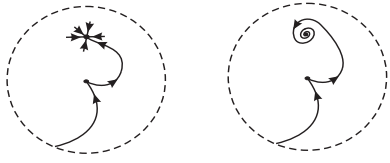}
\end{center}
\caption{Phase portrait of Case $1$: $a_{1}a_2\neq 0$.}
\label{figura31}
\end{figure}

In the case $a_2=0$ the vector field $X$ has four singularities
given by $(0,0,\pm 1)$ and $(0,\pm 1,0)$. The eigenvalues
associated to the last two singularities are $\displaystyle 0,\;(
a_1\pm\sqrt{a_1^2-4a_5^2})/2$ and $\displaystyle 0,\;-(
a_1\pm\sqrt{a_1^2-4a_5^2})/2$, respectively. Therefore, these
singularities can be nodes or foci. Now, substituting $a_2=0$ in
the above change of coordinates we obtain that in this case the
system associated to $X$ is equivalent to system \eqref{eq:154}
with $a_2=0$. Therefore, by the previous argument, we have that
$X$ does not have limit cycles on $\mathbb S^2$. Thus, the phase
portrait of $\tilde{X}$ on the Poincar\'e disc is equivalent to
the one of Figure \ref{figura32}.
\begin{figure}[ptb]
\begin{center}
\includegraphics[height=0.8in,width=1.8in]{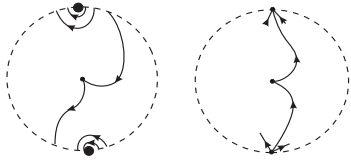}
\end{center}
\caption{Phase portrait of Case $1$: $a_{1}\neq 0$ and $a_2= 0$.}
\label{figura32}
\end{figure}

\smallskip\noindent {\it Case 2:} $a_{1}=0$. In this case we
have that (\ref{eq:121}) has the straight line of singularities
$v=0$ and, so $X$ has an invariant circle formed by singularities
determined by the intersection of the plane $y=0$ with $\mathbb
S^2$. Moreover, $H(x,y,z)=\displaystyle a_5y-a_2z$ is a first
integral of $X$ (note that $a_5\neq 0$) and $X$ has other two
singularities $(0,a_5, - a_2)/(a_2^2+a_5^2)$ and $(0,-a_5,
a_2)/(a_2^2+a_5^2)$. Then its phase portrait is equivalent to the
one of Figure \ref{figura2} $(a)$ if $a_2\neq 0$, or to the one
Figure \ref{figura2} $(b)$ if $a_2=0$. Hence, statement $(b)$ is
proved.

\begin{figure}[ptb]
\begin{center}
\psfrag{A}[l][B]{$(a)$} \psfrag{B}[l][B]{$(b)$}
\includegraphics[height=0.9in,width=2in]{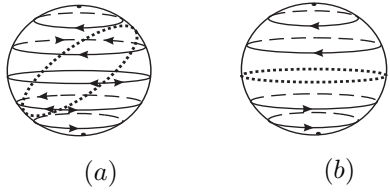}
\end{center}
\caption{Phase portrait of Case $2$.} \label{figura2}
\end{figure}

\smallskip Now, we prove statement $(c)$. By Theorem
\ref{the:08}, we can suppose that $X$ is in the form (\ref{eq:04})
with $a_3=a_6=0$, $a_8=-a_4$, $a_4^2+a_7a_5<0$ and
$a_4(a_1^2-a_2^2)+a_1a_2(a_5+a_7)=0$. Moreover, by the proof of
Theorem \ref{the:08}, for characterizing the phase portrait of $X$
we have to study only systems (\ref{eq:123}) and (\ref{eq:125}).
Therefore, we distinguish the following two cases.

\smallskip\noindent {\it Case 3:} $a_{1}=0$ and $a_2\neq 0$.
Then $a_4=0$. This case correspond to system \eqref{eq:123}. First
we suppose that $a_7\neq -a_5$. In this case $X$ has only two
singularities on $\mathbb S^1$, given by $(\pm 1, 0,0)$ having
eigenvalues $\{0, -(a_2\pm \sqrt{a_2^2-4(a_7^2+a_5a_7)})/2\}$ and
$\{0,(a_2\pm \sqrt{a_2^2-4(a_7^2+a_5a_7)})/2\}$, respectively. The
system induced by $X$ through the central projection, i.e. system
(\ref{eq:123}) has a center in $(0,0)$ and another singularity
$\left(0, a_5/a_2\right)$ with eigenvalues $ \pm
\sqrt{-(a_2^2+a_5^2)(a_5^2+a_5a_7)}/a_2$. Note that
\begin{equation}
\label{eq:127} (a_5^2+a_5a_7)(a_7^2+a_5a_7)=(a_5+a_7)^2a_5a_7<0,
\end{equation}
because $a_5a_7<0$ and $a_5+a_7\neq 0$. Therefore, if the
singularities on $\mathbb S^1$ are saddles, then
$\left(0,\displaystyle a_5/a_2 \right)$ is a center because system
$(\ref{eq:123})$ is invariant with respect to the change of
variables $(u,v,t)\mapsto (-u,v,-t)$. Now if
$\left(0,\displaystyle a_5/a_2 \right)$ is a saddle, then the
singularities of $X$ on $\mathbb S^1$ can be nodes or foci. Note
that if $a_2^2-4(a_7^2+a_5a_7)\geq 0$ then $\{f_{\pm}=0\}\cap
\mathbb S^2$, where
$f_{\pm}(x,y,z)=y+2a_7z/(-a_2\pm\sqrt{a_2^2-4(a_7^2+a_5a_7)})$,
are invariant circles of $X$ on $\mathbb S^2$ with cofactors $
K_{\pm}(x,y,z)=(-a_2\pm\sqrt{a_2^2-4(a_7^2+a_5a_7)})x/2$.
Therefore, in the case that the singularities of $X$ on $\mathbb
S^1$ are either saddles or nodes, then $X$ has at least an
invariant circle on $\mathbb S^2$. Hence, by Corollary
\ref{cor:01} $X$ has no limit cycles on $\mathbb S^2$. Thus, in
this case, the phase portrait of $X$ on the Poincar\'e disc is
equivalent to the one of Figures \ref{figura33} $(a)$ or $(b)$ on
$\mathbb S^2$.

Now, we consider the planar vector field induced by $X$ through
the stereographic projection (\ref{eq:136}) at the point (0,0,1),
i.e. we consider system \eqref{eq:139} which in this case becomes
\begin{equation}
\label{eq:141}
\begin{array}{lcl}
\dot{u} & = & -2a_5v+4a_2v^2-2(a_5+2a_7)u^2v+2a_5v^3, \\
\dot{v} & = & -2a_7u-4a_2uv-2(a_7+2a_5)uv^2+2a_7u^3.
\end{array}
\end{equation}
This system has an inverse integrating factor (see Remark
\ref{ob:03}) given by $V(u,v)=(u^2+v^2+1)V_2(u,v)$, where $
V_2(u,v)=a_7u^4+2(a_7(v^2-1)-a_2v)u^2+a_7(1+v^2)^2+2v(2a_5v+a_2(1-v^2))$.
Note that que the roots of $V_2=0$ in $u$ are
\[
u=\pm\frac{\sqrt{a_7(a_7v^2+a_2v+a_7\pm|v|\sqrt{a_2^2-4(a_7^2+a_5a_7)})}}{a_7}.
\]
Therefore, if $a_2^2-4(a_7^2+a_5a_7)<0$, $V_2$ vanish only on the
points $(\pm 1,0)$. Hence, by Remark \ref{ob:03}, system
\eqref{eq:141} has a first integral defined on $\mathbb
R^2\setminus \{(\pm 1,0)\}$. Hence, system \eqref{eq:141} does not
have limit cycles. Therefore, if the singularities of $X$ on
$\mathbb S^1$ are foci, then the phase portrait of $X$ on the
Poincar\'e disc is equivalent to Figure \ref{figura33} $(c)$.

\begin{figure}[ptb]
\begin{center}
\psfrag{A}[l][B]{$(a)$} \psfrag{B}[l][B]{$(b)$}
\psfrag{C}[l][B]{$(c)$}
\includegraphics[height=1in,width=2.8in]{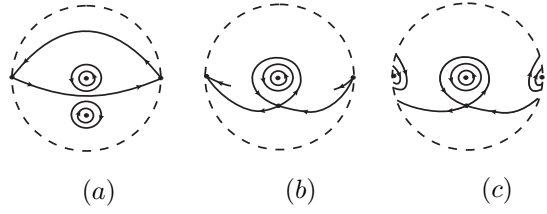}
\end{center}
\caption{Phase portrait of Case $3$: $a_{1}=a_4=0$, $a_2\neq 0$
and $a_7\neq -a_5$.} \label{figura33}
\end{figure}

Now if $a_7=-a_5$, then the system associated to $X$ becomes
\[
\dot{x}=y(a_2y+a_5z),\;\;\dot{y}=-x(a_2y+a_5z),\;\;\dot{z}=0.
\]
Therefore, the phase portrait of $X$ on $\mathbb S^2$ in this case
is equivalent to Figure \ref{figura2} $(a)$.

\smallskip\noindent {\it Case 4:} $a_{1}=a_2=0$. This case
correspond to system $(\ref{eq:125})$. We claim that we can reduce
to the case $a_4= 0$. Now we prove the claim. In fact, consider
the system associated to $X$ on $\mathbb S^2$, i.e. system
(\ref{eq:04}) with $a_{1}=a_2=a_3=a_6=0$, $a_8=-a_4$ and
$a_4^2+a_7a_5<0$. If $a_4\neq 0$ and $a_5+a_7\neq 0$, by the
orthogonal linear change of variables
\[
\left(\begin{array}{c} x \\ y
\\ z \end{array}\right)=\left(\begin{array}{ccc}

a & b & 0
\\  d &
e & 0 \\
0 & 0 & 1 \end{array} \right)\left(\begin{array}{c} \tilde{x} \\
\tilde{y}
\\ \tilde{z} \end{array}\right),
\]
where $a=a_4\sqrt{2\sigma}/((a_5+a_7)\sqrt{\delta\beta})$,
$b=(a_5(a_5+a_7)\delta-(2a_4^2+a_5(a_5+a_7))\sqrt{\sigma})/\sqrt{\delta\beta\gamma}$,
$d=-\sqrt{\beta}/\sqrt{2\delta}$,
$e=-(a_4((a_5+a_7)(4a_4^2+a_7(a_5+a_7)+a_5(a_5+a_7))+(a_7-a_5)\sqrt{\sigma}))\sqrt{\delta
\beta\gamma}$, $ \delta=(a_5+a_7)^2+4a_4^2$,
$\sigma=(a_5+a_7)^2\delta,\;\beta=\delta-\sqrt{\sigma}$, $
\gamma=(a_5+a_7)^2(a_5^2+a_7^2+2a_4^2)+(a_7^2-a_5^2)\sqrt{\sigma}$,
we obtain $\tilde{a}_4=0$. Here $\tilde{a}_i$ denote the
coefficients of system associated to $X$ on the variables
$(\tilde{x},\tilde{y},\tilde{z})$. Note that in this case $\delta>
0$, $\sigma>0$ and as
$(\delta-\sqrt{\sigma})(\delta+\sqrt{\sigma})=4a_4^2\delta$, $
((a_5+a_7)^2(a_5^2+a_7^2+2a_4^2)+(a_7^2-a_5^2)\sqrt{\sigma})\cdot((a_5+a_7)^2(a_5^2+a_7^2+2a_4^2)-(a_7^2-a_5^2)\sqrt{\sigma})=
4(a_5+a_7)^4(a_4^2+a_5a_7)^2$, it follows that $\beta>0$,
$\gamma>0$. Moreover, we have that
$\tilde{a}_1=\tilde{a}_2=\tilde{a}_3=\tilde{a}_4=\tilde{a}_6=\tilde{a}_8=0$,
$\tilde{a}_5=-\sqrt{2}(a_5+a_7)(a_5a_7+a_4^2)/\sqrt{\gamma}$ and
$\tilde{a}_7=-\sqrt{\gamma}/(\sqrt{2}(a_5+a_7))$. Note that
$\tilde{a}_5\tilde{a}_7<0$.

Now, if $a_4\neq 0$ and $a_7=-a_5$, then $a_4^2-a_5^2<0$ and doing
the orthogonal linear change of variables
\[
\left(\begin{array}{c} x \\ y
\\ z \end{array}\right)=\left(\begin{array}{ccc}
\displaystyle \frac{\sqrt{2a_4^2}}{2a_4} & \displaystyle
\frac{a_4+a_5}{\sqrt{2(a_5+a_4)^2}} & 0
\\  \displaystyle \frac{\sqrt{2a_4^2}}{2a_4} &
\displaystyle
-\frac{a_4+a_5}{\sqrt{2(a_5+a_4)^2}} & 0 \\
0 & 0 & 1 \end{array} \right)\left(\begin{array}{c} \tilde{x} \\
\tilde{y}
\\ \tilde{z} \end{array}\right)
\]
we have $\tilde{a}_5=a_4(a_4^2-a_5^2)/\sqrt{a_4^2(a_5+a_4)^2}$,
$\tilde{a}_7=\sqrt{a_4^2(a_5+a_4)^2}/a_4$,
$\tilde{a}_1=\tilde{a}_2=\tilde{a}_3=\tilde{a}_4=\tilde{a}_6=\tilde{a}_8=0$.
Note that $\tilde{a}_5\tilde{a}_7<0$. In short, the claim is
proved.

Now we study the case $a_4=0$ and $a_5+a_7\neq 0$, $X$ has six
singularities on $\mathbb S^2$ given by $(\pm 1,0,0)$, $(0,\pm
1,0)$ and $(0,0,\pm 1)$. The eigenvalues associated to the
singularities $(\pm 1,0,0)$ and $(0,\pm 1,0)$ are $\{0,\pm
\sqrt{-(a_7^2+a_5a_7)}\}$ and $\{ 0,\pm \sqrt{-(a_5^2+a_5a_7)}\}$,
respectively. As $a_5a_7<0$, from (\ref{eq:127}) in this case if
two of these opposite singularities are saddles the other two have
complexes eigenvalues and they are foci or centers. In fact, they
are centers. Suppose that $-(a_5^2+a_5a_7)<0$. By the change of
variables $(x,y,z)\mapsto (x,z,y)$ the system associated to $X$
becomes
\begin{equation}
\label{eq:152} \dot{x}  =  a_5zy, \;\; \dot{y}  =  -(a_5+a_7)zx,
\;\; \dot{z}  =  a_7yx.
\end{equation}
By Theorem \ref{the:08}, the singularities $(0,0,\pm 1)$ are
centers of the above system. Hence, $(0,\pm 1,0)$ are centers of
$X$. Similarly if $-(a_7^2+a_5a_7)<0$, by the change of variables
$(x,y,z)\mapsto (z,y,x)$, we conclude that $(\pm 1,0,0)$ are
centers of $X$. Hence, the phase portrait of $X$ on the Poincar\'e
disc is equivalent to Figure \ref{figura35}.

If $a_4=0$ and $a_7=-a_5$, then the system associated to $X$
becomes
\[
\dot{x}=a_5yz,\;\;\dot{y}=-a_5xz,\;\;\dot{z}=0.
\]
Therefore, the phase portrait of $X$ on $\mathbb S^2$ in this case
is equivalent to Figure \ref{figura2} $(b)$.

We note that the study of system \eqref{eq:152} is called by some
authors (see \cite{ref:20}) the Euler Problem.

\begin{figure}[ptb]
\begin{center}
\includegraphics[height=0.8in,width=0.8in]{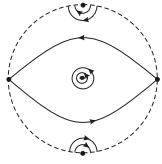}
\end{center}
\caption{Phase portrait of Case $4$: $a_{1}=a_2=0$ and $a_4=0$.}
\label{figura35}
\end{figure}

Now we prove statement $(d)$. If $X$ has a semi--hyperbolic
singularity then, by definition, for determining its phase
portrait we have to study system (\ref{eq:115}) with $a_4+a_8\neq
0$ and $a_4a_8-a_5a_7=0$. We can divided the proof in three cases.
First case $a_4=a_5=0$ and so $a_8\neq 0$. Second case $a_4=a_7=0$
and so $a_8\neq 0$. Third case $a_8=a_5a_7/a_4$ and $a_4\neq
-a_8$. However, we need study only the first case because the last
two cases are equivalent to it by a rotation of the variables. In
fact, if the coefficients of $X$ satisfy the second or the third
cases, then by the orthogonal linear change of variables
$(x,y,z)=M_i (\tilde{x},\tilde{y},\tilde{z})$, $i=1,2$, where
\[
M_1=\left(\begin{array}{ccc}\displaystyle
-\frac{a_8}{\sqrt{a_5^2+a_8^2}} & \displaystyle
\frac{a_5}{\sqrt{a_5^2+a_8^2}} & 0 \\
\displaystyle \frac{a_5}{\sqrt{a_5^2+a_8^2}} & \displaystyle
\frac{a_8}{\sqrt{a_5^2+a_8^2}} & 0 \\ 0 & 0 & 1
\end{array}\right)
\]
and
\[
M_2=\left(\begin{array}{ccc}\displaystyle
-\frac{a_7}{\sqrt{a_7^2+a_4^2}} & \displaystyle
\frac{a_4}{\sqrt{a_7^2+a_4^2}} & 0 \\
\displaystyle \frac{a_4}{\sqrt{a_7^2+a_4^2}} & \displaystyle
\frac{a_7}{\sqrt{a_7^2+a_4^2}} & 0 \\ 0 & 0 & 1
\end{array}\right),
\]
the coefficients of $X$ in the variables
$(\tilde{x},\tilde{y},\tilde{z})$ satisfy the first case,
respectively. Thus, we suppose that $a_4=a_5=0$ and so $a_8\neq
0$. We distinguish the following four cases.

\smallskip\noindent {\it Case 5:} $a_{1}a_7\neq 0$. We have
that $X$ has only four singularities on $\mathbb S^2$, i.e.
$(0,0,\pm 1)$,\linebreak
$(a_7/\sqrt{a_7^2+a_1^2},0,a_1/\sqrt{a_7^2+a_1^2})$ and
$(-a_7/\sqrt{a_7^2+a_1^2},0,-a_1/\sqrt{a_7^2+a_1^2})$. These
singularities correspond to the singularities $(0,0)$ and
$(-a_7/a_1,0)$ of system (\ref{eq:115}). Doing the change of
variables $u=-a_8r/a_7$, $v=r+s$ and introducing the new time
variable $\tau$ through $d\tau=-a_8dt$ system (\ref{eq:115}) can
be written as
\begin{equation}
\label{eq:128} \dot{r}=P_2(r,s),\;\;\dot{s}=s+Q_2(r,s).
\end{equation}
If $a_2a_7-a_1a_8\neq 0$, we have that $\varphi
(r)=((a_2a_7-a_1a_8)(a_7^2+a_8^2))r^2/(a_7^2a_8^2)+\cdots$ is the
solution of $s+Q_2(r,s)=0$. Hence, $\psi (r)=P(r,\varphi
(r))=(a_2a_7-a_1a_8)r^2/a_8^2+\cdots$. Thus, by Theorem $65$ of
page $340$ of \cite{ref:7}, we have that $(0,0)$ is a saddle--node
of system (\ref{eq:128}). Now, the trace and the determinant of
the linear part from system (\ref{eq:115}) at $(-a_7/a_1,0)$ are
$(a_2a_7-a_1a_8)/a_1$ and $a_7^2(a_7^2+a_1^2)/a_1^2$,
respectively. Therefore, $(-a_7/a_1,0)$ can be a node or a focus
of system (\ref{eq:115}). Now, by the orthogonal linear change of
variables
\[
\left(\begin{array}{c} x \\
y \\ z
\end{array}\right)=\left(\begin{array}{ccc} 0 & \displaystyle \frac{a_1}{\sqrt{a_1^2+a_7^2}} & \displaystyle \frac{a_7}{\sqrt{a_1^2+a_7^2}}
\\ 1 & 0 & 0 \\
0 & -\displaystyle \frac{a_7}{\sqrt{a_1^2+a_7^2}} & \displaystyle
\frac{a_1}{\sqrt{a_1^2+a_7^2}} \end{array}
\right)\left(\begin{array}{c} \tilde{x} \\ \tilde{y}
\\ \tilde{z} \end{array}\right),
\]
the system associated to $X$ becomes
\[
\begin{array}{lcl}
\dot{\tilde{x}} & = &
P(\tilde{x},\tilde{y},\tilde{z})=\displaystyle
-\frac{a_1a_2+a_7a_8}{\sqrt{a_1^2+a_7^2}}\tilde{x}\tilde{y}-\frac{a_2a_7-a_1a_8}{\sqrt{a_1^2+a_7^2}}\tilde{x}\tilde{z}-
a_1\tilde{y}^2-a_7\tilde{y}\tilde{z}, \\
\dot{\tilde{y}} & = & Q(\tilde{x},\tilde{y},\tilde{z})=
\displaystyle
\frac{a_1a_2+a_7a_8}{\sqrt{a_1^2+a_7^2}}\tilde{x}^2+a_1\tilde{x}\tilde{y}+
a_7\tilde{x}\tilde{z}, \\
\dot{\tilde{z}} & = & R(\tilde{x},\tilde{y},\tilde{z})=
\displaystyle \frac{a_2a_7-a_1a_8}{\sqrt{a_1^2+a_7^2}}\tilde{x}^2.
\end{array}
\]
Note that in these coordinates the singularity
$(-a_7/(\sqrt{a_1^2+a_7^2}),0,-a_1/(\sqrt{a_1^2+a_7^2}))$
corresponds to the point $(0,0,-1)$. Hence, using the notation of
Section \ref{sec:09}, we have that
$\tilde{R}(u,v)=R(2u,2v,u^2+v^2-1)=
4(a_2a_7-a_1a_8)/\sqrt{a_1^2+a_7^2}u^2$. Hence, $\tilde{R}$ does
not change sign and by Theorem \ref{the:11}, $X$ does not have
limit cycles on $\mathbb S^2$. So the phase portrait of $X$ on the
Poincar\'e disc is equivalent to one of Figure \ref{figura36}.
Note that the separatrices of the hyperbolic sectors of the
diametrally opposite saddles--nodes can be connect or not, see
Figures $5$ and $6$ of \cite{ref:15}.

\begin{figure}[ptb]
\begin{center}
\includegraphics[height=0.8in,width=2in]{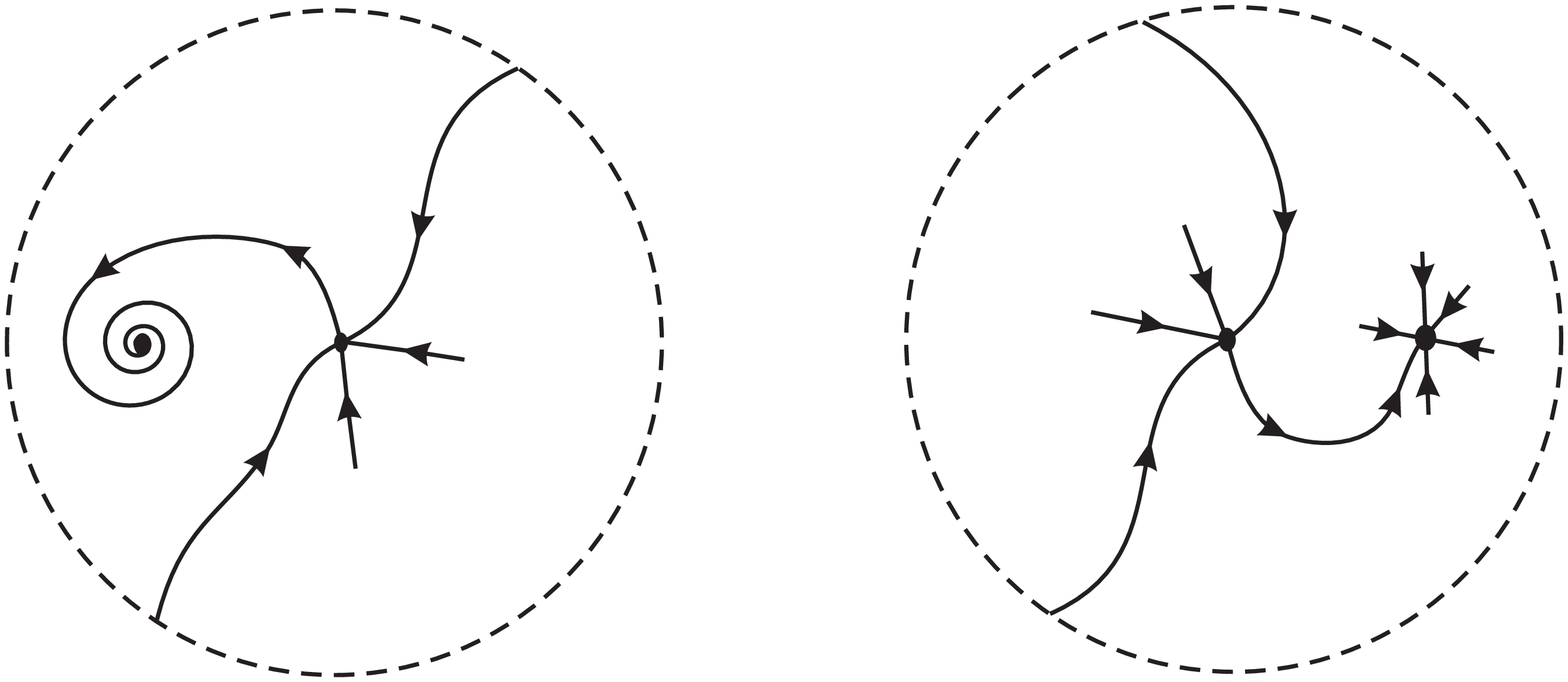}
\end{center}
\caption{Phase portrait of Case $5$: $a_{1}a_7\neq 0$ and
$a_7a_2-a_8a_1\neq 0$.} \label{figura36}
\end{figure}

Now if $a_7a_2-a_8a_1=0$, system (\ref{eq:128}) has a straight
line of singularities passing though the origin. Note that
$H(x,y,z)=a_7x+a_1z$ is a first integral of $X$ and $\{H=0\}\cap
\mathbb S^2$ is a circle of singularities. So the phase portrait
of $X$ is equivalent to Figure \ref{figura2} $(b)$.

\smallskip\noindent {\it Case 6:} $a_7=0$ and $a_1\neq 0$. So
$(0,0,\pm 1)$ are the unique singularities of $X$ on $\mathbb
S^2$. Introducing the new time variable $\tau$ through $d\tau
=-a_8dt$ system (\ref{eq:115}) can be written in the form
\[
\dot{u} = P_2(u,v),\;\;\dot{v}  = v+Q_2(u,v).
\]
We have that $v=\varphi (u)=-a_1u^2/2a_8+\cdots$ is the solution
of $v+Q_2(u,v)=0$. Hence, $\psi (u)=P(u,\varphi
(u))=a_1^2u^2/a_8^2+\cdots$. Thus, by Theorem $65$ of page $340$
of \cite{ref:7}, we have that $(0,0)$ is a topological node of
system (\ref{eq:115}).

Now we consider the planar vector field induced by $X$ through the
stereographic projection of (\ref{eq:136}) at the point $(0,0,1)$,
i.e. we consider system (\ref{eq:139}) which in this case becomes
\begin{equation}
\label{eq:146} \dot{r} = 4a_1rs+4a_2s^2-4a_8rs^2, \;\; \dot{s}  =
 -2a_8s-4a_1r^2-4a_2rs+2a_8r^2s-2a_8s^3.
\end{equation}
Since $\dot{s}\mid_{s=0}=-4a_1r^2$, the orbits of vector field
determined by system (\ref{eq:146}) intersect the straight line
$s=0$ always in a unique direction except at the point
$(r,s)=(0,0)$, i.e $s=0$ is transversal to the flow of system
(\ref{eq:146}) except at the origin. Thus, $X=(P,Q,R)$ does not
have limit cycles on $\mathbb S^2$. We also can apply Theorem
\ref{the:11} of Section \ref{sec:09} to give another proof of this
fact. Then, it is sufficient to see that
$\tilde{R}(u,v)=R(2u,2v,u^2+v^2-1)=-4a_8v^2$ does not change sign.
Hence, the phase portrait of $X$ on the Poincar\'e disc is
equivalent to Figure \ref{figura37}.

\begin{figure}[ptb]
\begin{center}
\includegraphics[height=0.8in,width=0.8in]{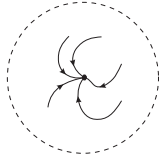}
\end{center}
\caption{Phase portrait of Case $6$: $a_7=0$ and $a_1\neq 0$.}
\label{figura37}
\end{figure}

\smallskip\noindent {\it Case 7:} $a_7=0$ and $a_1=0$. We have
that $\{y=0\}\cap \mathbb S^2$ is a circle of singularities of
$X$. As $H(x,y,z)=a_8x+a_2z$ is a first integral of $X$, it
follows that phase portrait of $X$ is equivalent to Figure
\ref{figura2} $(a)$.

\smallskip\noindent {\it Case 8:} $a_7\neq 0$ and $a_1=0$.
Assume that $a_2\neq 0$, then $X$ has only four singularities on
$\mathbb S^2$, $(0,0,\pm 1)$ and $(\pm 1,0,0)$. The singularities
$(0,0,\pm 1)$ correspond to the singularity $(0,0)$ of system
(\ref{eq:115}). Doing the change of variables $u=-a_8r/a_7$,
$v=r+s$ and introducing the new time variable $\tau$ through
$d\tau=-a_8dt$ system (\ref{eq:115}) can be written
\begin{equation}
\label{eq:129} \dot{r} =P_2(r,s), \;\; \dot{s}  =  s+Q_2(r,s).
\end{equation}
Since $a_2\neq 0$ we have that $s=\varphi
(r)=(a_2(a_7^2+a_8^2))r^2/(a_7a_8^2)+\cdots$ is the solution of
$s+Q_2(r,s)=0$. Hence, $\psi (r)=P(r,\varphi
(r))=(a_7a_2)r^2/a_8^2+\cdots$. Thus, by Theorem $65$ of page
$340$ of \cite{ref:7}, we have that $(0,0)$ is a saddle--node of
system (\ref{eq:129}). Now, the singularities $(\pm 1,0,0)$ have
eigenvalues $\{0,-(a_2\pm \sqrt{a_2^2-4a_7^2})/2$ and $\{0,(a_2\pm
\sqrt{a_2^2-4a_7^2})/2$, respectively. Therefore, $(\pm 1,0,0)$
can be nodes or foci of $X=(P,Q,R)$ restricted to $\mathbb S^2$.
Now, using the notation of Section \ref{sec:09} and doing the
change of variables $(x,y,z)\longmapsto (z,y,x)$, we have that $X$
becomes $(\mathcal{P},\mathcal{Q},\mathcal{R})=(R(z,y,x),
Q(z,y,x), P(z,y,x))$ and so
$\tilde{\mathcal{R}}(u,v)=P(u^2+v^2-1,2v,2u)=4a_2v^2$. Hence,
$\tilde{\mathcal{R}}$ does not change sign and by Theorem
\ref{the:11}, $X$ does not have limit cycles on $\mathbb S^2$. So
the phase portrait of $X$ on the Poincar\'e disc is equivalent to
one of Figures \ref{figura38}. Note that the separatrices of the
hyperbolic sectors of the diametrally opposite saddles--nodes can
be connect or not, see Figures $5$ and $6$ of \cite{ref:15}.

\begin{figure}[ptb]
\begin{center}
\includegraphics[height=1.8in,width=2in]{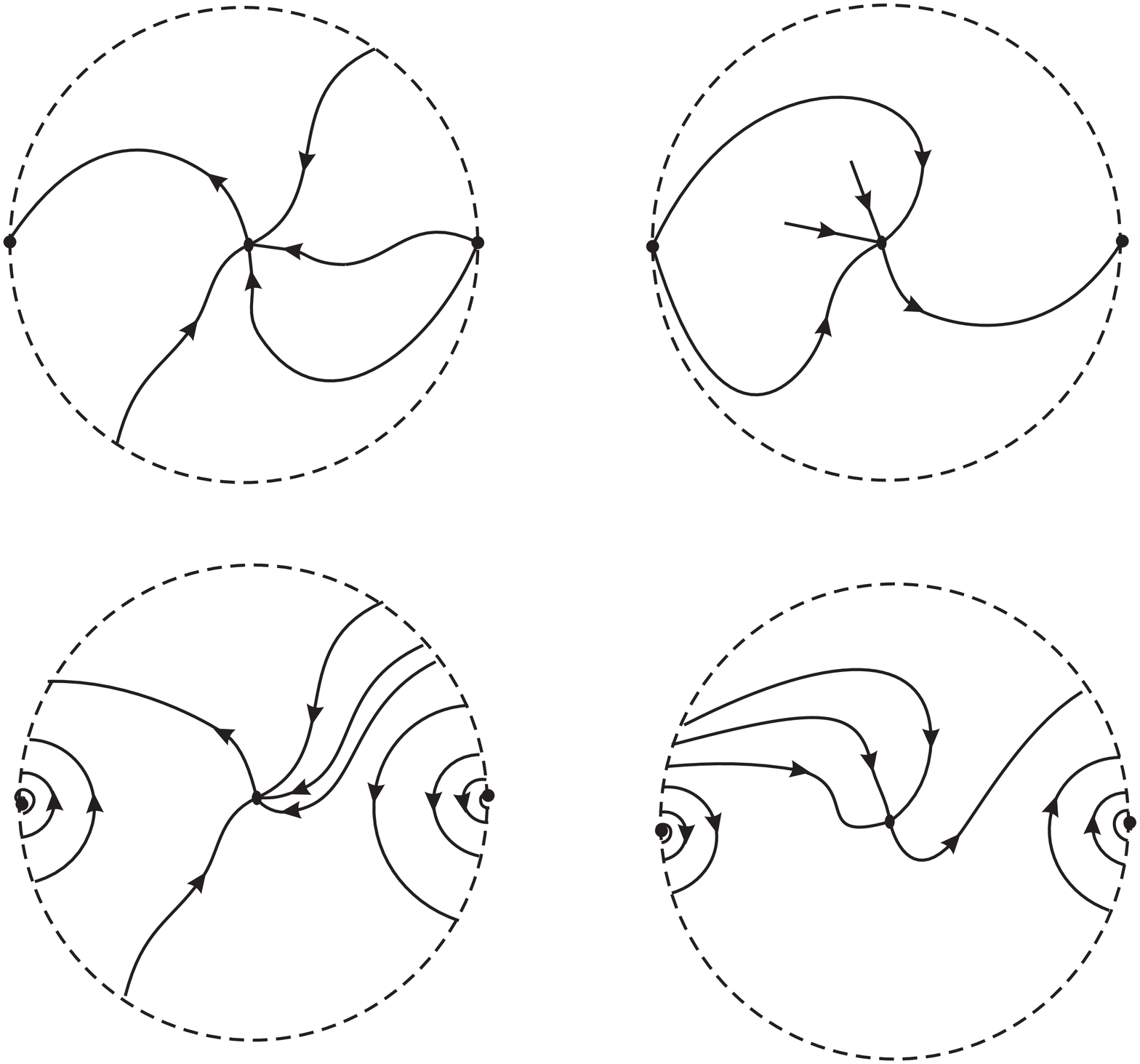}
\end{center}
\caption{Phase portrait of Case $8$: $a_7a_2\neq 0$ and $a_1=0$.}
\label{figura38}
\end{figure}

Now if $a_2=0$, $X$ has a circle of singularities on $\mathbb S^2$
determined by the intersection of the plane $a_7x+a_8y=0$ with
$\mathbb S^2$. Note that $H(x,y,z)=x$ is a first integral of $X$.
So the phase portrait of $X$ is equivalent to Figure \ref{figura2}
$(a)$. \hfill $\Box$

\smallskip\noindent {\it Proof of Corollary \ref{cor:02}}. It
is a straightforward consequence of the proof of Theorem
\ref{the:09}. \hfill$\Box$

\section{Singular points of homogeneous polynomial vector fields of $\mathbb {CP}^2$}
\label{sec:11}

In this section we present arguments and results
which go back to Darboux \cite{ref:8}, see also \cite{ref:21}.
First we recall that $\mathbb {CP}^2=\{\mathbb C^3\setminus
\{(0,0,0)\}\}/\thicksim$ with the equivalence relation
$[X,Y,Z]\thicksim [X',Y',Z']$ if, and only if, there exists
$\lambda\in \mathbb C\setminus \{0\}$ such that
$[X',Y',Z']=\lambda [X,Y,Z]$.

Let $P$, $Q$ and $R$ be homogeneous polynomials of degree $m+1$ in
the complex variables $X$, $Y$ and $Z$. We say that the
homogeneous $1$--form $\omega =PdX+QdY+RdZ$ of degree $m+1$ is
{\it projective} if $XP+YQ+ZR=0$. Therefore, the $1$--form is
projective if and only if there exist three homogeneous
polynomials $L$, $M$ and $N$ of degree $m$ such that $
P=ZM-YN,\;Q=XN-ZL,\;R=YL-XM$, or equivalently $
(P,Q,R)=(L,M,N)\land (X,Y,Z)$. Then we can write
\[
\omega =\left|\begin{array}{ccc}L & M & N \\ X & Y & Z \\ dX & dY
& dZ
\end{array}\right|.
\]
We notice that the polynomials $L$, $M$ and $N$ are not uniquely
determined by $P$, $Q$ and $R$. These polynomials can be replaced
by $L'=L+X\Delta$, $M'=M+Y\Delta$ and $N'=N+Z\Delta$, where
$\Delta$ is any homogeneous polynomial at the variables $X$, $Y$,
$Z$ with degree $m-1$.

The projective $1$--form $\omega$ defines a differential equation
in $\mathbb{CP}^2$ given by $\omega =0$, which may written as $
(ZM-YN)dX+(XN-LZ)dY+(LY-MX)dZ=0$.

Usually in the literature a projective $1$--form $\omega$ is
called a {\it Pfaff algebraic form} of degree $m+1$ of $\mathbb
{CP}^2$, see for more details Jouanolou \cite{ref:17}.

Let $F(X,Y,Z)$ be a homogeneous polynomial of degree $n$ in the
variables $X$, $Y$, $Z$ and coefficients in $\mathbb C$. Then
$F(X,Y,Z)=0$ is an algebraic curve of $\mathbb {CP}^2$. Let
$p=(X_0,Y_0,Z_0)$ be a point of $\mathbb {CP}^2$. Since the three
coordinates of $p$ cannot be zero, without loss of generality we
can assume that $p=(0,0,-1)$. Then suppose that the expression of
$F(X,Y,Z)$ restricted to $Z=-1$ is
\[
F(X,Y,-1)=F_d(X,Y)+F_{d+1}(X,Y)+\cdots +F_n(X,Y),
\]
where $0\leq d\leq n$ and $F_j(X,Y)$ denotes a homogeneous
polynomial of degree $j$ in the variables $X$ and $Y$ for
$j=d,\ldots n$, with $F_d$ different from the zero polynomial. We
say that $d=d_p(F)$ is the {\it multiplicity} of the curve $F=0$
at the point $p$. If $d=0$ then the point $p$ does not belong to
the curve $F=0$. If $d=1$ we say that $p$ is a {\it simple} point
for the curve $F=0$. If $d>1$ we say that $p$ is a {\it multiple}
point of $F=0$. In particular $p$ is a multiple point of $F=0$ if
and only if
\[
\frac{\partial F}{\partial X}(p)=\frac{\partial F}{\partial
Y}(p)=\frac{\partial F}{\partial Z}(p)=0.
\]
Suppose that for $d>0$ we have that $F_d=\Pi_{i=1}^rL_i^{r_i}$ are
different straight lines, called {\it tangent lines} to $F=0$ at
the point $p$, $r_i$ is the {\it multiplicity} of the tangent line
$L_i$ at $p$. For $d>1$ we say that $p$ is an {\it ordinary
multiple} point if the multiplicity of all tangents at $p$ is $1$,
otherwise we say that $p$ is a {\it non--ordinary multiple point}.

Let $F=0$ and $G=0$ be two algebraic curves and $p$ a point of
$\mathbb {CP}^2$. We say that $F=0$ and $G=0$ intersect {\it
strictly} at $p$, if $F=0$ and $G=0$ do not have any common
component which pass through $p$. We say that $F=0$ and $G=0$
intersect {\it transversally} at $p$ if $p$ is a simple point of
$F=0$ and $G=0$, and the tangent to $F=0$ at $p$ is distinct to
the tangent to $G=0$ at $p$. The proof of the following two
theorems can be found in \cite{ref:13}.

\begin{theorem}
{\rm \bf (Intersection Number Theorem)} There exists a unique
multiplicity or intersection number $I(p,F\cap G)$ defined for all
algebraic curves $F=0$ and $G=0$ and for all point $p$ of $\mathbb
{CP}^2$ satisfying the following properties.
\begin{itemize}
\item[(a)] $I(p,F\cap G)$ is a non--negative integer for all $F$,
$G$ and $p$ when $F=0$ and $G=0$ intersect strictly at $p$.
$I(p,F\cap G)=\infty$ if $F=0$ and $G=0$ do not intersect strictly
at $p$. \item[(b)] $I(p,F\cap G)=0$ if and only if $p$ is not a
common point to $F=0$ and $G=0$. $I(p,F\cap G)$ depends only on
the components of $F=0$ and $G=0$ which pass through $p$.
\item[(c)] If $T$ is a change of coordinates and $T(p)=q$, then
$I(q,T(F)\cap T(G))=I(p,F\cap G)$. \item[(d)] $I(p,F\cap
G)=I(p,G\cap F)$. \item[(e)] $I(p,F\cap G)\geq d_p(F)d_p(G)$, the
equality holds if and only if $F=0$ and $G=0$ do not have common
tangents at $p$. \item[(f)] $I(p,F\cap G_1G_2)=I(p,F\cap
G_1)+I(p,F\cap G_2)$. \item[(g)] $I(p,F\cap G)=I(p,F\cap (G+AF))$
for all homogeneous polynomial $A(X,Y,Z)$.
\end{itemize}
\end{theorem}

\begin{theorem}
{\rm \bf (Bezout Theorem)} Let $F=0$ and $G=0$ be two algebraic
curves of $\mathbb {CP}^2$ of degrees $r$ and $s$ respectively
without common components. Then $\sum_p I(p,F\cap G)=rs$.
\end{theorem}

Let $F_i=0$, for $i=1,\ldots, s$, be algebraic curves and $p$ a
point of $\mathbb {CP}^2$. We define the {\it multiplicity} or
{\it number of intersection} of the curves $F_1=0,\ldots,F_s=0$ at
$p$ as $I(p,\cap_{i=1}^sF_i)=\min_{i<j}\{I(p,F_i\cap F_j)\}$.

A point $[X_0,Y_0,Z_0]$ belonging to $\mathbb {CP}^2$ is a {\it
singular} point of a projective $1$--form $\omega=PdX+QdY+RdZ$ of
degree $m+1$ if $P(X_0,Y_0,Z_0)=Q(X_0,Y_0,Z_0)=R(X_0,Y_0,Z_0)=0$,
i.e. $[X_0,Y_0,Z_0]$ satisfies the system
\[
ZM-YN=0,\;\; XN-ZL=0,\;\; YL-XM=0.
\]
The proof of the next two results can be found in \cite{ref:8} or
\cite{ref:18}.

\begin{lemma} {\rm \bf (Darboux Lemma)} Let $A$, $A'$, $B$, $B'$,
$C$ and $C'$ be homogeneous polynomials in the variables $X$, $Y$
and $Z$ of degrees $l$, $l'$, $m$, $m'$, $n$ and $n'$
respectively, verifying the identity $AA'+BB'+CC'\equiv 0$. Assume
that the curves $A=0$, $B=0$, $C=0$ and the curves $A'=0$, $B'=0$,
$C'=0$ do not have any common component, respectively. Then $
\sum_p I(p,A\cap B\cap C)+\sum_p I(p,A'\cap B'\cap
C')=(lmn+l'm'n')/\gamma$, where $\gamma =l+l'=m+m'=n+n'$.
\end{lemma}

\begin{proposition}{\rm \bf (Darboux Proposition)}
For any projective $1$--form of degree $m+1$ of $\mathbb {CP}^2$
having finitely many singular points we have that its number of
singular points taking into account their multiplicities or
numbers of intersection satisfies $\sum_p I(p, (ZM-YN)\cap
(XN-ZL)\cap (YL-XM))=m^2+m+1$.
\end{proposition}

\smallskip\noindent {\it Proof of Theorem \ref{pro:12}.} We
have that $X$ induce a projective $1$--form $\omega$ on $\mathbb
{CP}^2$ of degree $n$ given by $\omega = PdX+QdY+RdZ$. Therefore,
by the Darboux Proposition, $\omega$ has $(n-1)^2+(n-1)+1=n^2-n+1$
singular points, taking into account their multiplicities or
numbers of intersection. Now, if $[X_0,Y_0,Z_0]$ is a singular
point of $\omega$ and $(X_0,Y_0,Z_0)\in \mathbb R^3$, then
$\lambda (X_0,Y_0,Z_0)$, $\lambda \in \mathbb R$, is a straight
line of singular points of $X$ which determine exactly two
singularities of $X$ on $\mathbb S^2$. Hence, we can conclude that
if $X$ has finitely many singularities on $\mathbb S^2$, then $X$
has at most $2(n^2-n+1)$ singularities on $\mathbb S^2$. \hfill
$\Box$

\section{Topological classification of all phase portraits of homogeneous
polynomial vector fields on $\mathbb S^{2}$ of degree $2$ modulo
limit cycles}

\label{sec:07}
 Let $X=(P,Q,R)$ be a homogeneous polynomial vector fields on
$\mathbb S^{2}$ of degree $2$. Consider the system (\ref{eq:04})
associated to $X$. Using the notation of previous section, we can
write (\ref{eq:04}) as
\begin{equation}
\label{eq:131}
\begin{array}{lcl}
\dot{x} & = & z(a_4x+a_5y+a_3z)-y(-a_1x-a_2y)=zM-yN, \\
\dot{y} & = & x(-a_1x-a_2y)-z(-a_7x-a_8y-a_6z)=xN-zL, \\
\dot{z} & = & y(-a_7x-a_8y-a_6z)-x(a_4x+a_5y+a_3z)=yL-xM,
\end{array}
\end{equation}
where $L(x,y,z)=-a_7x-a_8y-a_6z$, $M(x,y,z)=a_4x+a_5y+a_3z$,
$N(x,y,z)=-a_1x-a_2y$, i.e.
\begin{equation}
\label{eq:130} \left(\begin{array}{c} L \\ M
\\ N \end{array}\right)=\left(\begin{array}{ccc}
 -a_7 & -a_8 & -a_6 \\ a_4 & a_5 & a_3 \\
-a_1 & -a_2 & 0\end{array} \right)\left(\begin{array}{c} x \\ y
\\ z \end{array}\right)=A\left(\begin{array}{c} x \\ y \\ z \end{array}\right).
\end{equation}
Note that $X=(P,Q,R)=(L,M,N)\land (x,y,z)$. Then,
$(x_0,y_0,z_0)\neq (0,0,0)$ is a singular point of $X$ if and only
if there exist $\lambda_0 \in \mathbb R$ such that
$(L,M,N)(x_0,y_0,z_0)=\lambda_0 (x_0,y_0,z_0)$. Now, as $L$, $M$
and $N$ are homogeneous polynomials, it follows that $\lambda
(x_0,y_0,z_0)$, $\lambda \in \mathbb R$, is a straight line of
singularities of $X$. Therefore, for determining the singularities
of $X$ we must find the real eigenvectors of the $3\times 3$
matrix $A$ given by (\ref{eq:130}), i.e. we have to calculate the
real eigenvalues of $A$. Let $\lambda_0 \in \mathbb R$ an
eigenvalue of $A$. If there exists a unique eigenvector associated
to $\lambda_0$, then it determines a straight line of
singularities of $X$, and so we have two singularities of $X$ on
$\mathbb S^2$. Now if there exist two eigenvectors associated to
$\lambda_0$, then we have a plane of singularities of $X$ that
determine a circle of singularities of $X$ on $\mathbb S^2$. Thus,
$X$ has finitely many singularities on $\mathbb S^2$ if and only
if each real eigenvalue of $A$ has exactly one eigenvector
associated. Therefore, for degree $2$, we obtain another proof of
Theorem \ref{pro:12}. We also conclude that if $X$ has infinitely
many singularities on $\mathbb S^2$, then it has one invariant
circle on $\mathbb S^2$ formed by singularities and its phase
portrait is equivalent to one of Figures \ref{figura2} or
\ref{figura3}.

In Section \ref{sec:06} we classified all phase portraits of
homogeneous polynomial vector fields on $\mathbb S^2$ of degree
$2$ having a degenerate singularity or a center. Therefore, for
determining all phase portraits of a homogeneous polynomial vector
field $X$ on $\mathbb S^2$ of degree $2$ we have to consider only
the $X$'s with all their singularities non--degenerate and without
centers. In this case, the singularities can be saddles, nodes or
foci. Now, as these singularities have index $-1$ or $1$ and the
maximum number of singularities of $X$ is $6$, by symmetry of $X$
with respect to origin and the Poincar\'e--Hopf Index Theorem
(see, for instance \cite{ref:14}) it follows that we have to study
the phase portraits with two saddles and four singularities that
can be nodes or foci, and the phase portraits with two
singularities that can be either nodes or foci.

Let $p$ be a singularity of $X$, by Proposition \ref{pro:09}, we
can suppose that $p=(0,0,-1)$ and so $a_3=a_6=0$ in
(\ref{eq:131}). Hence,
\[
A=\left(\begin{array}{ccc}
 -a_7 & -a_8 & 0 \\ a_4 & a_5 & 0 \\
-a_1 & -a_2 & 0\end{array} \right).
\]
We have that the eigenvalues of $A$ are $\{0,(a_5-a_7\pm
\sqrt{\alpha})/2\}$, where $\alpha=(a_5+a_7)^2-4a_4a_8$.

In the case that $X$ has only two singularities that can be either
nodes or foci, then $A$ has either a unique eigenvalue with a
unique eigenvector associated, or $A$ has one real eigenvalue and
two complex eigenvalue. In the first case $0$ is the unique
eigenvalue of $A$ and
$0=(a_5-a_7+\sqrt{\alpha})(a_5-a_7-\sqrt{\alpha})=4(a_4a_8-a_5a_7)$,
i.e. $a_4a_8-a_5a_7=0$. Hence, using the notation of Section
\ref{sec:06}, we have that $(0,0,-1)$ is a degenerate singularity,
because $\det(D\widetilde{X}(0,0))=a_4a_8-a_5a_7=0$. This case was
already studied in Section \ref{sec:06}.

If $A$ has one real eigenvalue and two complex, then $0$ is the
real eigenvalue of $A$ and $\alpha<0$. Hence, as in this case the
system associated to $X$ has the normal form \eqref{eq:04} with
$a_3=a_6=0$, we have that
$\tilde{R}(u,v)=R(2u,2v,u^2+v^2-1)=-4(a_4u^2 + (a_5 + a_7)uv +
a_8v^2)$. Solving $\tilde{R}=0$ with respect to $u$, it follows
that $ u=-(a_5+a_7\pm\sqrt{\alpha})v/(2a_4)$ are its roots. Note
that, as $\alpha<0$, $a_4\neq 0$. Thus we have that $(0,0)$ is the
unique real root of $\tilde{R}=0$, and that $\tilde{R}$ does not
change of sign. Therefore, applying Theorem \ref{the:11} $(a)$, we
have that in this case $X$ has no limit cycles and so its phase
portrait in the Poincar\'e disc is equivalent to the one of Figure
\ref{figura37}.

We have that $(0,0,-1)$ is a saddle of $X$ if and only if
$\det(D\widetilde{X}(0,0))<0$, i.e. $a_4a_8-a_5a_7<0$. In this
case, as $(a_5+a_7)^2-4a_4a_8=(a_5-a_7)^2-4(a_4a_8-a_5a_7)>0$, $A$
has three reals eigenvalues given by $\{0,(a_5-a_7\pm
\sqrt{\alpha})/2\}$. So $\alpha$ must be positive. Therefore, the
eigenvectors of $A$ associated to the eigenvalues $\{0,(a_5-a_7\pm
\sqrt{\alpha})/2\}$ are $(0,0,1)$,
$((a_8(a_5-a_7+\sqrt{\alpha}))/(2a_1a_8-a_2(a_5+a_7+\sqrt{\alpha})),(
a_4(a_5-a_7+\sqrt{\alpha}))/(2a_2a_4-
a_1(a_5+a_7-\sqrt{\alpha})),1)$ and
$((a_8(a_5-a_7-\sqrt{\alpha}))/(2a_1a_8-a_2(a_5+a_7-\sqrt{\alpha})),
(a_4(a_5-a_7-\sqrt{\alpha}))/(2a_2a_4-a_1(a_5+a_7+\sqrt{\alpha})),1)$,
respectively. Note that the eigenvectors of $A$ are well defined
if
$(2a_1a_8-a_2(a_5+a_7+\sqrt{\alpha}))(2a_1a_8-a_2(a_5+a_7-\sqrt{\alpha}))(2a_2a_4-a_1(a_5+a_7+\sqrt{\alpha}))(2a_2a_4-a_1(a_5+a_7-\sqrt{\alpha}))\neq
0$, i.e. if $16a_4a_8(a_2^2a_4-a_1a_2(a_5+a_7)+a_1^2a_8)^2\neq 0$.
System (\ref{eq:131}) with $a_3=a_6=0$ and $a_4a_8-a_5a_7<0$ have
two saddles and four singularities that can be nodes, foci or
centers if and only if $A$ has three linearly independent
eigenvectors. Thus, we distinguish the following cases.

\smallskip\noindent {\it Case 1:}
$16a_4a_8(a_2^2a_4-a_1a_2(a_5+a_7)+a_1^2a_8)^2\neq 0$. This case
has been classified in the previous paragraph.

\smallskip\noindent {\it Case 2:} $a_4=0$, $a_1\neq 0$ and
$a_2(a_5+a_7)-a_1a_8\neq 0$. In this case $-a_5a_7<0$ and the
eigenvalues of $A$ are $0$, $-a_7$ and $a_5$ with respective
eigenvectors $(0,0,1)$,
$(a_5a_8/(a_2(a_5+a_7)-a_1a_8),-a_5(a_5+a_7)/(a_2(a_5+a_7)-a_1a_8),1)$
and $(a_7/a_1,0,1)$.

\smallskip\noindent {\it Case 3:} $a_4=a_1=0$ and $a_2\neq 0$.
In this case $-a_5a_7<0$ and so $a_5+a_7\neq 0$. The eigenvalues
of $A$ are $0$, $-a_7$ and $a_5$ with respective eigenvectors
$(0,0,1)$, $(1,0,0)$ and $(a_5a_8/(a_2(a_5+a_7)),-a_5/a_2,1)$.

\smallskip\noindent {\it Case 4:} $a_4=a_1=a_2=0$. In this
case $-a_5a_7<0$ and so $a_5+a_7\neq 0$. The eigenvalues of $A$
are $0$, $-a_7$ and $a_5$ with respective eigenvectors $(0,0,1)$,
$(1,0,0)$ and $(-a_8/(a_5+a_7),1,0)$.

\smallskip\noindent {\it Case 5:} $a_4=0$, $a_1\neq 0$ and
$a_8=a_2(a_7+a_5)/a_1$. Now the eigenvalues of $A$ are $0$, $-a_7$
and $a_5$ with respective eigenvectors $(0,0,1)$, $(a_7/a_1,0,1)$
and $(-a_2/a_1,1,0)$.

\smallskip\noindent {\it Case 6:} $a_1=a_2=0$ and $a_4a_8\neq
0$. The eigenvalues of $A$ are $0$, $(a_5+a_7\pm \sqrt{\alpha})/2$
with respective eigenvectors $(0,0,1)$,
$(-2a_8/(a_5+a_7+\sqrt{\alpha}),1,0)$ and
$(-2a_8/(a_5+a_7-\sqrt{\alpha}),1,0)$.

\smallskip\noindent {\it Case 7:}
$a_8=a_2(-a_2a_4+a_1(a_5+a_7))/a_1^2$ and $a_1a_4\neq 0$. The
eigenvalues of $A$ are $0$, $(a_2a_4-a_1a_7)/a_1$ and
$(a_1a_5-a_2a_4)/a_1$ with respective eigenvectors $(0,0,1)$,
$(((a_1a_7-a_2a_4)(a_1(a_5+a_7)-a_2a_4))/(a_1^2(a_1(a_5+a_7)-2a_2a_4)),
((a_2a_4-a_1a_7)a_4)/(a_1(a_1(a_5+a_7)-2a_2a_4)),1)$ and
$(-a_2/a_1,1,0)$. Note that
$a_4a_8-a_5a_7=(a_2a_4-a_1a_5)(a_1a_7-a_2a_4)/a_1^2<0$ and
$a_1(a_5+a_7)-2a_2a_4\neq 0$. Otherwise,
$a_5=(2a_2a_4-a_1a_7)/a_1$ and so
$a_4a_8-a_5a_7=(a_1a_7-a_2a_4)^2/a_1^2>0$.

\smallskip\noindent {\it Case 8:} $a_8=0$. We have
$-a_5a_7<0$. This case is equivalent to one of Cases $2$--$5$,
i.e. if the coefficients of system \eqref{eq:131} with $a_3=a_6=0$
satisfies $a_8=0$ then  by the orthogonal linear change of
variables
\[
\left(\begin{array}{c} x \\ y
\\ z \end{array}\right)=\left(\begin{array}{ccc}
 0 & 1 & 0 \\ 1 & 0 & 0 \\
0 & 0 & 1 \end{array} \right)\left(\begin{array}{c} \tilde{x} \\
\tilde{y}
\\ \tilde{z} \end{array}\right)
\]
we obtain one of Cases $2$--$5$.

Case $1$ is equivalent to Case $2$ doing the orthogonal linear
change of variables $(x,y,z)=M(\tilde{x}, \tilde{y}, \tilde{z})$,
where
\[
M=\left(\begin{array}{ccc} \displaystyle
-\frac{2a_8}{\sqrt{(a_5+a_7+\sqrt{\alpha})^2+4a_8^2}} &
\displaystyle
\frac{a_5+a_7+\sqrt{\alpha}}{\sqrt{(a_5+a_7+\sqrt{\alpha})^2+4a_8^2}}
& 0
\\ \displaystyle \frac{a_5+a_7+\sqrt{\alpha}}{\sqrt{(a_5+a_7+\sqrt{\alpha})^2+4a_8^2}} & \displaystyle \frac{2a_8}{\sqrt{(a_5+a_7+\sqrt{\alpha})^2+4a_8^2}} & 0 \\
0 & 0 & 1 \end{array} \right),
\]
and the $a_i$'s denote the coefficients of system \eqref{eq:131}
in the variables $(x,y,z)$ satisfying Case $1$. Indeed, denoting
by $\tilde{a}_i$ the coefficients of system \eqref{eq:131} in the
variables $(\tilde{x},\tilde{y},\tilde{z})$, we have that
$\tilde{a}_1=(2a_1a_8-a_2(a_5+a_7+\sqrt{\alpha}))/\sqrt{(a_5+a_7+\sqrt{\alpha})^2+4a_8^2}$,
$\tilde{a}_2=-(2a_2a_8+a_1(a_5+a_7+\sqrt{\alpha}))/\sqrt{(a_5+a_7+\sqrt{\alpha})^2+4a_8^2}$,
$\tilde{a}_5=1/2(a_7-a_5+\sqrt{\alpha})$,
$\tilde{a}_7=1/2(a_5-a_7+\sqrt{\alpha})$, $\tilde{a}_8=a_4+a_8$,
and $\tilde{a}_3=\tilde{a}_4=\tilde{a}_6=0$. Now, as
$16a_4a_8(a_2^2a_4-a_1a_2(a_5+a_7)+a_1^2a_8)^2\neq 0$, it follows
that $\tilde{a}_4=0$, $\tilde{a}_1\neq 0$,
$\tilde{a}_2(\tilde{a}_5+\tilde{a}_7)-\tilde{a}_1\tilde{a}_8\neq
0$ and $-\tilde{a}_5\tilde{a}_7=a_4a_8-a_5a_7<0$.

Case $6$ is equivalent to Case $4$ if $a_5+a_7\neq 0$ doing the
orthogonal linear change of variables
\[
\left(\begin{array}{c} x \\ y
\\ z \end{array}\right)=\left(\begin{array}{ccc}

a & b & 0
\\  d &
e & 0 \\
0 & 0 & 1 \end{array} \right)\left(\begin{array}{c} \tilde{x} \\
\tilde{y}
\\ \tilde{z} \end{array}\right),
\]
where
$a=((a_5+a_7)^2(a_4+a_8)+(a_4-a_8)\sqrt{\sigma})/((a_5+a_7)\sqrt{2\delta\beta})$,
$b =
((a_5+a_7)(a_5(a_5+a_7)^2-a_4(a_8(3a_5+a_7)+a_4(a_7-a_5)))-(a_4(a_4-a_8)+a_5(a_5+a_7))\sqrt{\sigma})/(\sqrt{\delta\beta\gamma}$,
$d=-\sqrt{\beta}/\sqrt{2\delta}$,
$e=((a_5+a_7)(a_4(2a_8(a_4-a_8)-a_7(a_5+a_7))+a_5a_8(a_5+a_7))-(a_7a_4+a_5a_8)\sqrt{\sigma})/\sqrt{\delta
\beta\gamma}$, $\delta=(a_5+a_7)^2+(a_4-a_8)^2$,
$\sigma=(a_5+a_7)^2\alpha$,
$\beta=(a_5+a_7)^2+2a_4(a_4-a_8)-\sqrt{\sigma}$,
$\gamma=(a_5+a_7)^2(a_5^2+a_7^2-2a_4a_8)+(a_7^2-a_5^2)\sqrt{\sigma}$,
and the $a_i$'s denote the coefficients of system \eqref{eq:131}
in the variables $(x,y,z)$ satisfying Case $6$. Note that in this
case $\delta> 0$, $\sigma>0$. Now, as
$((a_5+a_7)^2+2a_4(a_4-a_8)-\sqrt{\sigma})\cdot((a_5+a_7)^2+2a_4(a_4-a_8)+\sqrt{\sigma})=4a_4^2\delta$,
$((a_5+a_7)^2(a_5^2+a_7^2-2a_4a_8)+(a_7^2-a_5^2)\sqrt{\sigma})\cdot((a_5+a_7)^2(a_5^2+a_7^2-2a_4a_8)-(a_7^2-a_5^2)\sqrt{\sigma})=
4(a_5+a_7)^4(a_4a_8-a_5a_7)^2$,
$(a_5+a_7)^2+2a_4(a_4-a_8)=a_5^2+a_7^2+2a_4^2-2(a_4a_8-a_5a_7)>0$
and $a_5^2+a_7^2-2a_4a_8=(a_5-a_7)^2-2(a_4a_8-a_5a_7)>0$ it
follows that $\beta>0$ and $\gamma>0$. Denoting by $\tilde{a}_i$
the coefficients of system \eqref{eq:131} in the variables
$(\tilde{x},\tilde{y},\tilde{z})$, we have that
$\tilde{a}_5=-\sqrt{2}(a_5+a_7)(a_5a_7-a_4a_8)/\sqrt{\gamma}$,
$\tilde{a}_7=-\sqrt{\gamma}/(2(a_5+a_7))$, $\tilde{a}_8=a_4+a_8$,
$\tilde{a}_1=\tilde{a}_2=\tilde{a}_4=\tilde{a}_3=\tilde{a}_4=\tilde{a}_6=0$.
Moreover, $-\tilde{a}_5\tilde{a}_7<0$.

Now, Case $6$ is equivalent to Case $4$ if $a_7=-a_5$ doing the
orthogonal linear change of variables
\[
\left(\begin{array}{c} x \\ y
\\ z \end{array}\right)=\left(\begin{array}{ccc}
\displaystyle \frac{\sqrt{a_8^2-a_4a_8}}{a_4-a_8} & \displaystyle
\frac{a_4\sqrt{a_8^2-a_4a_8}+a_5\sqrt{a_4^2-a_4a_8}}{\sqrt{(a_5(a_4-a_8)+\sqrt{-(a_4-a_8)^2a_4a_8})^2}}
& 0
\\  \displaystyle \frac{\sqrt{a_4^2-a_4a_8}}{a_4-a_8} &
\displaystyle
\frac{a_8\sqrt{a_4^2-a_4a_8}-a_5\sqrt{a_8^2-a_4a_8}}{\sqrt{(a_5(a_4-a_8)+\sqrt{-(a_4-a_8)^2a_4a_8})^2}} & 0 \\
0 & 0 & 1 \end{array} \right)\left(\begin{array}{c} \tilde{x} \\
\tilde{y}
\\ \tilde{z} \end{array}\right).
\]
Indeed, denoting by $\tilde{a}_i$ the coefficients of system
\eqref{eq:131} in the variables $(\tilde{x},\tilde{y},\tilde{z})$,
we have that
$\tilde{a}_5=(a_5^2+a_8a_4)(a_8-a_4)/\sqrt{(a_5(a_4-a_8)+\sqrt{-(a_4-a_8)^2a_4a_8})^2}$,
$\tilde{a}_7=\sqrt{(a_5(a_4-a_8)+\sqrt{-(a_4-a_8)^2a_4a_8})^2}/(a_4-a_8)$,
$\tilde {a}_8=a_4+a_8$ and
$\tilde{a}_1=\tilde{a}_2=\tilde{a}_3=\tilde{a}_4=\tilde{a}_6=0$.
Now, since $a_5^2+a_8a_4<0$, it follows that
$-\tilde{a}_5\tilde{a}_7<0$.

Case $3$ is equivalent to Case $5$ doing the orthogonal linear
change of variables
\[
\left(\begin{array}{c} x \\ y
\\ z \end{array}\right)=\left(\begin{array}{ccc}
\displaystyle -\frac{a_8}{\sqrt{(a_5+a_7)^2+a_8^2}} &
\displaystyle \frac{a_5+a_7}{\sqrt{(a_5+a_7)^2+a_8^2}} & 0
\\ \displaystyle \frac{a_5+a_7}{\sqrt{(a_5+a_7)^2+a_8^2}} & \displaystyle \frac{a_8}{\sqrt{(a_5+a_7)^2+a_8^2}} & 0 \\
0 & 0 & 1 \end{array} \right)\left(\begin{array}{c} \tilde{x} \\
\tilde{y}
\\ \tilde{z} \end{array}\right),
\]
where the $a_i$'s denote the coefficients of system \eqref{eq:131}
in the variables $(x,y,z)$ satisfying Case $3$. Indeed, denoting
by $\tilde{a}_i$ the coefficients of system \eqref{eq:131} in the
variables $(\tilde{x},\tilde{y},\tilde{z})$, we have that
$\tilde{a}_1=-a_2(a_5+a_7)/\sqrt{(a_5+a_7)^2+a_8^2}$,
$\tilde{a}_2=-a_2a_8/\sqrt{(a_5+a_7)^2+a_8^2}$, $\tilde{a}_5=a_7$,
$\tilde{a}_7=a_5$, $
\tilde{a}_8=\tilde{a}_2(\tilde{a}_5+\tilde{a}_7)/\tilde{a}_1$ and
$\tilde{a}_3=\tilde{a}_4=\tilde{a}_6=0$. Now, as $a_2\neq 0$ and
$-a_5a_7<0$, it follows that $\tilde{a}_1\neq 0$ and
$-\tilde{a}_5\tilde{a}_7< 0$.

Case $7$ is equivalent to Case $5$ doing the orthogonal linear
change of variables $(x,y,z)=M(\tilde{x},\tilde{y},\tilde{z})$
with $M$ given by
\[
\left(\begin{array}{ccc} \displaystyle
\frac{a_2a_4-a_1(a_5+a_7)}{\sqrt{(a_1(a_5+a_7)-a_2a_4)^2+a_1^2a_4^2}}
& \displaystyle
\frac{a_4a_1}{\sqrt{(a_1(a_5+a_7)-a_2a_4)^2+a_1^2a_4^2}} & 0
\\ \displaystyle \frac{a_4a_1}{\sqrt{(a_1(a_5+a_7)-a_2a_4)^2+a_1^2a_4^2}} & \displaystyle -\frac{a_2a_4-a_1(a_5+a_7)}{\sqrt{(a_1(a_5+a_7)-a_2a_4)^2+a_1^2a_4^2}} & 0 \\
0 & 0 & 1 \end{array} \right),
\]
where the $a_i$'s denote the coefficients of system \eqref{eq:131}
in the variables $(x,y,z)$ satisfying Case $3$. Note that, as
$a_1\neq 0$ and
$a_4a_8-a_5a_7=(a_2a_4-a_1a_5)(a_1a_7-a_2a_4)/a_1^2<0$, it follows
that $(a_1(a_5+a_7)-a_2a_4)^2+a_1^2a_4^2\neq 0$. Denoting by
$\tilde{a}_i$ the coefficients of system \eqref{eq:131} in the
variables $(\tilde{x},\tilde{y},\tilde{z})$, we have that
$\tilde{a}_1=(a_1(a_5+a_7)-2a_2a_4)a_1/\sqrt{(a_1(a_5+a_7)-a_2a_4)^2+a_1^2a_4^2}$,
$\tilde{a}_2=(a_4(a_2^2-a_1^2)-a_1a_2(a_5+a_7))/\sqrt{(a_1(a_5+a_7)-a_2a_4)^2+a_1^2a_4^2}$,
$\tilde{a}_5=(a_2a_4-a_1a_5)/a_1$,
$\tilde{a}_7=(a_2a_4-a_1a_7)/a_1$,
$\tilde{a}_8=\tilde{a}_2(\tilde{a}_5+\tilde{a}_7)/\tilde{a}_1$,
and  $\tilde{a}_3=\tilde{a}_4=\tilde{a}_6=0$. Now, as
$a_1(a_5+a_7)-2a_2a_4\neq 0$, it follows that $\tilde{a}_1\neq 0$.
Moreover,
$-\tilde{a}_5\tilde{a}_7=(a_2a_4-a_1a_5)(a_1a_7-a_2a_4)/a_1^2< 0$.

We conclude that for determining all phase portraits of all
homogeneous polynomial vector field on $\mathbb S^2$ of degree $2$
having two saddles and four singularities that can be nodes, foci
or centers we have to consider only Case $2$, $4$ and $5$, i.e. we
have to consider the following families
\begin{equation}
\label{eq:142}
\begin{array}{lcl}
\dot{x} & = & a_5yz-y(-a_1x-a_2y), \\
\dot{y} & = & x(-a_1x-a_2y)-z(-a_7x-a_8y), \\
\dot{z} & = & y(-a_7x-a_8y)-a_5xy,
\end{array}
\end{equation}
with $a_1\neq 0$, $a_2(a_5+a_7)-a_1a_8\neq 0$ and $-a_5a_7<0$;

\begin{equation}
\label{eq:143} \dot{x}  =  b_5yz, \;\; \dot{y} = -z(-b_7x-b_8y),
\;\; \dot{z}  =  y(-b_7x-b_8y)-b_5xy,
\end{equation}
with $-b_5b_7<0$; and

\begin{equation}
\label{eq:144}
\begin{array}{lcl}
\dot{x} & = & c_5yz-y(-c_1x-c_2y), \\
\dot{y} & = & x(-c_1x-c_2y)-z\left(-c_7x-\displaystyle \frac{c_2(c_5+c_7)}{c_1}y\right), \\
\dot{z} & = & y\left(-c_7x-\displaystyle
\frac{c_2(c_5+c_7)}{c_1}y\right)-c_5xy,
\end{array}
\end{equation}
with $c_1\neq 0$ and $-c_5c_7<0$.

Consider the vector field $X$ determined by system (\ref{eq:143}).
Doing the orthogonal linear change of variables
\[
\left(\begin{array}{c} x \\ y
\\ z \end{array}\right)=\left(\begin{array}{ccc}
0 & 0  & 1
\\ 0 & 1 & 0 \\
1 & 0 & 0 \end{array} \right)\left(\begin{array}{c} \bar{x} \\
\bar{y}
\\ \bar{z} \end{array}\right),
\]
system (\ref{eq:143}) becomes
\begin{equation}
\label{eq:153} \dot{\bar{x}}  =
-b_8\bar{y}^2-(b_5+b_7)\bar{y}\bar{z}, \;\;\dot{\bar{y}}  =
b_8\bar{x}\bar{y}+b_7\bar{x}\bar{z}, \;\; \dot{\bar{z}}  =
b_5\bar{x}\bar{y}.
\end{equation}
Note that for $a_2=-b_8$, $a_5=-(b_5+b_7)$ and $a_7=b_7$ system
(\ref{eq:153}) becomes the system of Case $3$ in the proof of
Theorem \ref{the:09} of Section \ref{sec:06} with
$a_7^2+a_5a_7<0$.

Now we will study the possible connections of separatrices of
families (\ref{eq:142}) and (\ref{eq:144}) on $\mathbb S^2$. We
have consider only the homoclinic orbits, because as the saddles
of families (\ref{eq:142}) and (\ref{eq:144}) on $\mathbb S^2$ are
diametrally opposite, by Proposition \ref{pro:06}, if $X$ has a
heteroclinic orbit on $\mathbb S^2$, then $X$ has an invariant
circle on $\mathbb S^2$. Then, the phase portraits of homogeneous
polynomial vector fields on $\mathbb S^2$ of degree $2$ having
invariant circles have been classified in \cite{ref:15}.

We saw that families (\ref{eq:142}) and (\ref{eq:144}) have six
singularities, two saddles and other four singularities that can
be nodes, foci or centers. Therefore, the homoclinic orbits that
can appear in families (\ref{eq:142}) and (\ref{eq:144}) on
$\mathbb S^2$ must be one of the described in Figures
\ref{figura39} (a) and (c).

\begin{figure}[ptb]
\begin{center}
\psfrag{A}[l][B]{$(a)$} \psfrag{B}[l][B]{$(b)$}
\psfrag{C}[l][B]{$(c)$}
\includegraphics[height=0.6in,width=3in]{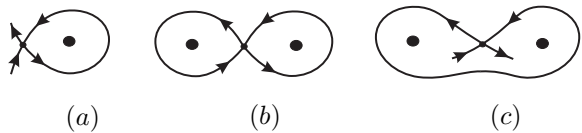}
\end{center}
\caption{Homoclinic orbits. The symbol $\bullet$ denotes a focus,
a center or a node} \label{figura39}
\end{figure}

\begin{proposition}
Let $X$ be a homogeneous polynomial vector field of degree $2$.
Then $X$ cannot realize Figure \ref{figura39} $(b)$ or $(c)$.
\end{proposition}
\begin{proof}
Suppose that $X$ realizes some of Figures \ref{figura39} $(b)$ or
$(c)$. We can suppose without loss of generality that the three
singularities of Figure \ref{figura39} $(b)$ or $(c)$ are in the
southern hemisphere. We consider a great circle on $\mathbb S^2$
that contains the saddle and another great circle that contains
the other two singularities. The great circle through the saddle
is chosen in a such way that the two straight lines determined by
these great circles through the central projection (\ref{eq:82})
in the tangent plane to $\mathbb S^2$ at the point $(0,0,-1)$ are
parallel. Eventually these two great circles can coincide. Now,
there exists a different great circle $C$ such that the straight
line determined by it is also parallel to the previous straight
lines, it does not contain the saddle point and the other two
singularities belong to the half--sphere determined by it and
containing the saddle point. As the orbits of $X$ are symmetric
with respect to the origin, these three invariant circles cannot
intersect in two singular opposite points. Moreover, we can choose
$C$ such that the straight line determined by it is sufficiently
close to the straight line that contains the saddle point, and $C$
intersects transversally the homoclinic orbits in four points (see
Figure \ref{figura40}). Note that, as the orbits of $X$ are
symmetric with respect to the origin, $C$ has six tangent points
with orbits of $X$. But this is a contradiction with Proposition
\ref{pro:14}. Therefore, the phase portrait of $X$ cannot realize
Figure \ref{figura39} $(b)$ or $(c)$.
\end{proof}

\begin{figure}[ptb]
\begin{center}
\psfrag{B}[l][B]{$C$}
\includegraphics[height=0.8in,width=2in]{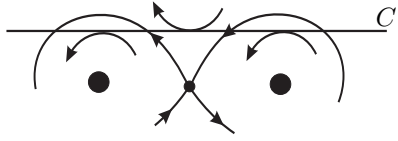}
\end{center}
\caption{The study of cases (b) and (c) of Figure \ref{figura39}}
\label{figura40}
\end{figure}

{From} Figure \ref{figura33} we can see that there exist
homogeneous polynomial vector fields on $\mathbb S^2$ of degree
$2$ having a homoclinic orbit of the type described in Figure
\ref{figura39} $(a)$, where the singularity inside the loop is a
center. However, we do not have examples where this singularity is
a focus or a node, but numerical computation show that they exist.
For example, we can consider the following family of homogeneous
polynomial vector fields on $\mathbb S^2$ of degree $2$ determined
by the system
\begin{equation}
\label{eq:156} \dot{x} = -xy-2y^2+yz, \;\;\dot{y}  =
x^2+2xy+xz+a_8yz, \;\;\dot{z}  =  -2xy-a_8y^2,
\end{equation}
with $a_8\neq 4$. System \eqref{eq:156} is obtained from
\eqref{eq:142} substituting $a_1=-1$, $a_2=-2$, $a_5=1$ and
$a_7=1$. Now, consider the system induced by \eqref{eq:156}
through the central projection on the tangent plane to $\mathbb
S^2$ at $(1,0,0)$, i.e. system \eqref{eq:103} becomes
\begin{equation}
\label{eq:157}
\begin{array}{lcl}
\dot{u} & = & -v-uv-2v^2-2u^2v-a_8uv^2, \\
\dot{v} & = & -u-a_8v+u^2+2uv-2uv^2-a_8v^3.
\end{array}
\end{equation}
This system has three singularities, $(0,0)$, $(1,0)$ and
$(-a_8/(a_8-4),2/(a_8-4))$ with respective eigenvalues $(-a_8\pm
\sqrt{a_8^2+4})/2$, $(2-a_8\pm \sqrt{(2-a_8)^2-16})/2$ and
$(a_8^2-2a_8+2\pm
\sqrt{(a_8^2-2a_8+2)^2-8(a_8^2+(a_8-4)^2+4)})/(2(4-a_8))$. Using
the software $P4$, we can obtain the phase portraits of system
\eqref{eq:157} in the Poincar\'e disc. Thus, for $a_8=8/5$ and
$a_8=9/5$ we obtain the phase portraits given by Figures
\ref{figura42} $(a)$ and $(b)$, respectively.
\begin{figure}[ptb]
\begin{center}
\psfrag{B}[l][B]{$(a)$} \psfrag{C}[l][B]{$(b)$}
\includegraphics[height=0.8in,width=1.7in]{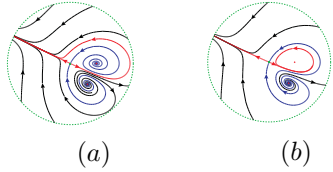}
\end{center}
\caption{Phase portrait of system \eqref{eq:157}.}
\label{figura42}
\end{figure}
In Figure \ref{figura42} we have that $(0,0)$ is a saddle and
$(1,0)$, $(-a_8/(a_8-4),2/(a_8-4))$ are unstable foci. Note that
in Figure \ref{figura42} $(a)$ one of the stable separatrices of
$(0,0)$ emerges from the unstable focus $(1,0)$. Now, in Figure
\ref{figura42} $(b)$ one of the unstable separatrices of $(0,0)$
tend to the unstable focus $(1,0)$. Therefore, there exists at
least one limit cycle surrounding $(1,0)$. We can conclude by
Figure \ref{figura42} that there exists a value of the parameter
$a_8$ belonging to the interval $(8/5,9/5)$ such that there is a
homoclinic orbit surrounding the unstable focus $(1,0)$. In this
homoclinic orbit borns a limit cycle because in Figure
\ref{figura42} $(b)$ the point $(1,0)$ is an unstable focus
surrounded by an unstable separatrix coming from the saddle
$(0,0)$. It seems that moving the parameter of the system this
limit cycle ends in the Hopf bifurcation which takes place at the
singular point $(1,0)$.

We can conclude that if $X$ is a homogeneous polynomial vector
field of degree $2$ on $\mathbb S^2$ with all its singularities
non--degenerate, then its phase portrait on $\mathbb S^2$ is
topologically equivalent to one of the phase portraits on $\mathbb
S^2$ given in Figures \ref{figura41}. This proves Theorem
\ref{the:12}.

\begin{figure}[ptb]
\begin{center}
\includegraphics[height=0.8in,width=2.9in]{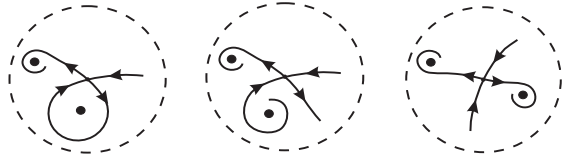}
\end{center}
\caption{Phase portrait on the Poincar\'e disc of homogeneous
polynomial vector fields on $\mathbb S^2$ of degree $2$ with a
saddle. The symbol $\bullet$ denotes either a focus or a node
surrounded perhaps by some limit cycles or a center. Note that the
three singularities also can be surrounded by some limit cycles.}
\label{figura41}
\end{figure}

\section{Hopf bifurcation on homogeneous polynomial vector fields on $\mathbb S^2$ of degree $2$}
\label{sec:12} Let $X=(P,Q,R)$ be a homogeneous polynomial vector
fields on $\mathbb S^2$ of degree $2$ and let $p$ be a singularity
of $X$. Without loss of generality we can suppose that
$p=(0,0,-1)$, because we can do a rotation of $SO(3)$ which
preserves all the properties of $X$. Hence, using the notation of
Section \ref{sec:06}, the singularity $p$ is a {\it weak focus} of
$X$ on $\mathbb S^2$ if $\mbox{tr}(D\tilde{X}(0,0))=0$,
$\mbox{det}(D\tilde{X}(0,0))>0$ and $p$ is not a center.
\begin{proposition}
\label{pro:15} Let $X$ be a homogeneous polynomial vector vector
field on $\mathbb S^2$ and let $p=(0,0,-1)$ a singularity of $X$.
Suppose that the system associated to $X$ is in the form
\eqref{eq:04} with $a_3=a_6=0$. Then $p$ is a weak focus of $X$ if
and only if $a_8=-a_4$, $a_4^2+a_7a_5<0$ and
$a_4(a_1^2-a_2^2)+a_1a_2(a_5+a_7)\neq 0$. Moreover, $p$ is stable
if $a_4(a_1^2-a_2^2)+a_1a_2(a_5+a_7)<0$ and unstable if
$a_4(a_1^2-a_2^2)+a_1a_2(a_5+a_7)>0$.
\end{proposition}
\begin{proof}The proof of this proposition follows directly of the proof of Theorem
\ref{the:08}.
\end{proof}

Before stating the next theorem, we will state the Hopf
Bifurcation Theorem.

\begin{theorem}{\rm \bf{(Hopf Bifurcation Theorem)}} Suppose that the
analytical parame-\linebreak trized system $\dot{x}=X(x,\mu)$,
$x\in \mathbb R^2$, $\mu \in \mathbb R$, has a singular point at
the origin for all values of the real parameter $\mu$.
Furthermore, suppose that the eigenvalues, $\lambda_1(\mu)$ and
$\lambda_2(\mu)$ of $DX(0,\mu)$, are pure imaginary for
$\mu=\mu_0$. If the real part of the eigenvalues,
$\mbox{Re}\lambda_1(\mu)=\mbox{Re}\lambda_2(\mu)$ in a
neighborhood of $\mu_0$, satisfies $\displaystyle
\frac{d}{d\mu}(\mbox{Re}\lambda_1(\mu))\mid_{\mu=\mu_0}>0$ and the
origin is an asymptotically stable fixed point when $\mu=\mu_0$,
then
\begin{itemize}
\item[(a)] $\mu=\mu_0$ is a bifurcation point of the system;
\item[(b)] for $\mu \in (\mu_1,\mu_0)$, some $\mu_1<\mu_0$, the
origin is a stable focus; \item[(c)] for $\mu \in (\mu_0,\mu_2)$,
some $\mu_2>\mu_0$, the origin is an unstable focus surrounded by
a stable limit cycle, whose size increases with $\mu$.
\end{itemize}
\end{theorem}
The proof of Hopf Bifurcation Theorem can be found in
\cite{ref:23}.

\smallskip\noindent {\it Proof of Theorem \ref{the:14}.}
Consider the vector field $X=(P,Q,R)$ of the beginning of this
section. By Sections \ref{sec:06} and \ref{sec:07}, it follows
that if $X$ has a weak focus, then the system associated to $X$ is
equivalent to one of systems \eqref{eq:142} or \eqref{eq:144}.

Consider system \eqref{eq:144} with $c_1\neq 0$ and $-c_5c_7<0$.
This system has six singularities, $(0,0,\pm 1)$,
$(c_7/(\sqrt{c_1^2+c_7^2}),0,c_1/(\sqrt{c_1^2+c_7^2}))$,
$(-c_7/(\sqrt{c_1^2+c_7^2}),0, -c_1/(\linebreak
\sqrt{c_1^2+c_7^2}))$ and two singularities on $\mathbb S^1$,
given by $(-c_2/(\sqrt{c_2^2+c_1^2}), c_1/(\sqrt{c_2^2+c_1^2}),0)$
and $(c_2/(\sqrt{c_2^2+c_1^2}), -c_1/(\sqrt{c_2^2+c_1^2}),0)$ with
respective eigenvalues $\{
0,(\sqrt{c_2^2+c_1^2}\pm\sqrt{c_2^2+c_1^2-4(c_5^2+c_5c_7)})/2\}$,
$ \{
0,(-\sqrt{c_2^2+c_1^2}\pm\sqrt{c_2^2+c_1^2-4(c_5^2+c_5c_7)})/2\}$.
As $c_5c_7\linebreak>0$, the singularities of $X$ on $\mathbb S^1$
are either nodes or strong foci on $\mathbb S^2$.

The planar system determined by system \eqref{eq:144} through the
central projection on the tangent plane to $\mathbb S^2$ at the
point $(0,0,-1)$ is given by
\[
\begin{array}{lcl}
\dot{u} & = & \displaystyle
-c_5v+c_1uv+c_2v^2-(c_5+c_7)u^2v-\frac{c_2(c_5+c_7)}{c_1}uv^2, \\
\dot{v} & = & \displaystyle
-c_7u-\frac{c_2(c_5+c_7)}{c_1}v-c_1u^2-c_2uv-(c_5+c_7)uv^2-\frac{c_2(c_5+c_7)}{c_1}v^3,
\end{array}
\]
see \eqref{eq:103}. This system has two singularities $(0,0)$ and
$(-c_7/c_1,0)$ corresponding to the singularities of $X$ that does
not belong to $\mathbb S^1$. We saw, in Section \ref{sec:07}, that
$(0,0)$ is a saddle. Now,
$(-c_2c_5\pm\sqrt{c_2^2c_5^2-4(c_7^2+c_5c_7)(c_1^2+c_7^2)})/2c_1$
are the eigenvalues of $(-c_7/c_1,0)$. Again, as $c_5c_7>0$, a
necessary condition in order that $(-c_7/c_1,0)$ be a weak focus
is $c_2=0$. If $c_2=0$, by the orthogonal linear change of
variables
\[
\left(\begin{array}{c} \tilde{x} \\ \tilde{y}
\\ \tilde{z} \end{array}\right)=\left(\begin{array}{ccc}
0 & 1 & 0
\\ 1 & 0 & 0 \\
0 & 0 & 1 \end{array} \right)\left(\begin{array}{c} x \\
y \\ z \end{array}\right),
\]
system (\ref{eq:144}) becomes
\[
\dot{\tilde{x}} =  -c_1\tilde{y}^2+c_7\tilde{y}\tilde{z}, \;\;
\dot{\tilde{y}}  =  c_1\tilde{x}\tilde{y}+c_5\tilde{x}\tilde{z},
\;\; \dot{\tilde{z}}  =  -(c_5+c_7)\tilde{x}\tilde{y}.
\]
This system satisfies Case $3$ in the proof of Theorem
\ref{the:09}. Therefore, $(-c_7/\sqrt{c_1^2+c_7^2},0,\linebreak
-c_1/(\sqrt{c_1^2+c_7^2})$ is a center of $X$ on $\mathbb S^2$ and
the phase portrait of $X$ on $\mathbb S^2$ is equivalent to the
phase portrait given by Figure \ref{figura33} $(b)$ or $(c)$.
Thus, in the family (\ref{eq:144}) its singularities cannot be
weak foci and so it does not have a Hopf bifurcation.

Now we consider system (\ref{eq:142}). System \eqref{eq:142} has
six singularities, $(0,0,\pm 1)$,
$(a_5a_8/\sqrt{\sigma},-a_5(a_5+a_7)/\sqrt{\sigma},
(a_2(a_5+a_7)-a_1a_8)/\sqrt{\sigma})$,
$(-a_5a_8/\sqrt{\sigma},a_5(a_5+a_7)/\sqrt{\sigma},\linebreak
-(a_2(a_5+a_7)-a_1a_8)/\sqrt{\sigma})$,
$(a_7/\sqrt{a_7^2+a_1^2},0,$ $a_1/\sqrt{a_7^2+a_1^2})$ and
$(-a_7/\sqrt{a_7^2+a_1^2},0,\linebreak -a_1/\sqrt{a_7^2+a_1^2})$,
where $\sigma = a_5^2(a_8^2+(a_5+a_7)^2)+(a_2(a_5+a_7)-a_1a_8)^2$.
We saw in Section \ref{sec:07} that $(0,0,\pm 1)$ are saddles and
the other four singularities can be nodes, foci or centers. By the
orthogonal linear change of variables
\[
\left(\begin{array}{c} \tilde{x} \\ \tilde{y}
\\ \tilde{z} \end{array}\right)=\left(\begin{array}{ccc}
0 & 1 & 0
\\ \displaystyle \frac{a_1}{\sqrt{a_1^2+a_7^2}} & 0 & \displaystyle -\frac{a_7}{\sqrt{a_1^2+a_7^2}} \\
\displaystyle \frac{a_7}{\sqrt{a_1^2+a_7^2}} & 0 & \displaystyle \frac{a_1}{\sqrt{a_1^2+a_7^2}} \end{array} \right)\left(\begin{array}{c} x \\
y \\ z \end{array}\right),
\]
system (\ref{eq:142}) becomes
\[
\begin{array}{lcl}
\dot{\tilde{x}} & = & \displaystyle
-\frac{a_1a_2+a_7a_8}{\sqrt{a_1^2+a_7^2}}\tilde{x}\tilde{y}-\frac{a_2a_7-a_1a_8}{\sqrt{a_1^2+a_7^2}}\tilde{x}\tilde{z}-
a_1\tilde{y}^2-a_7\tilde{y}\tilde{z}, \\
\dot{\tilde{y}} & = & \displaystyle \frac{a_1a_2+a_7a_8}{\sqrt{a_1^2+a_7^2}}\tilde{x}^2+a_1\tilde{x}\tilde{y}+(a_5+a_7)\tilde{x}\tilde{z}, \\
\dot{\tilde{z}} & = & \displaystyle
\frac{a_2a_7-a_1a_8}{\sqrt{a_1^2+a_7^2}}\tilde{x}^2-a_5\tilde{x}\tilde{y}.
\end{array}
\]
Note that in these coordinates the singularity
$(-a_7/(\sqrt{a_1^2+a_7^2}),0,-a_1/(\sqrt{a_1^2+a_7^2}))$ goes
over the point $(0,0,-1)$. Let $
\tilde{a}_1=-(a_1a_2+a_7a_8)/\sqrt{a_1^2+a_7^2}$,
$\tilde{a}_2=-a_1$,
$\tilde{a}_4=-(a_2a_7-a_1a_8/\sqrt{a_1^2+a_7^2}$,
$\tilde{a}_5=-a_7$, $\tilde{a}_7=a_5+a_7$, and $\tilde{a}_8=0$. By
Theorem \ref{the:08},
$(-a_7/\sqrt{a_7^2+a_1^2},0,-a_1/\sqrt{a_7^2+a_1^2})$ is a center
of system \eqref{eq:142} on $\mathbb S^2$ if $\tilde{a}_4=0$, i.e.
$a_2a_7-a_1a_8=0$,
$\tilde{a}_4^2+\tilde{a}_5\tilde{a}_7=-(a_7^2+a_5a_7)<0$ and $
\tilde{a}_4(\tilde{a}_1^2-\tilde{a}_2^2)+\tilde{a}_1\tilde{a}_2(\tilde{a}_5+\tilde{a}_7)=\sqrt{a_1^2+a_2^2}a_1a_5a_8/a_7=0$.
However, this is not possible, because $a_1a_5a_7\neq 0$,
$a_2a_7-a_1a_8=(\sqrt{a_1^2+a_2^2}a_1a_5a_8)/a_7=0$ implies
$a_2=a_8=0$, which is in contradiction with the hypothesis
$a_2(a_5+a_7)-a_1a_8\neq 0$. Therefore,
$(-a_7/\sqrt{a_7^2+a_1^2},0,-a_1/\sqrt{a_7^2+a_1^2})$ cannot be a
center. Now, if $a_2a_7-a_1a_8=0$, by Proposition \ref{pro:15} it
is a weak focus of system \eqref{eq:142}.

The planar system determined by system \eqref{eq:142} through the
central projection on the tangent plane to $\mathbb S^2$ at the
point $(0,0,-1)$ is
\begin{equation}
\label{eq:145}
\begin{array}{lcl}
\dot{u} & = & \displaystyle
-a_5v+a_1uv+a_2v^2-(a_5+a_7)u^2v-a_8uv^2, \\
\dot{v} & = & \displaystyle
-a_7u-a_8v-a_1u^2-a_2uv-(a_5+a_7)uv^2-a_8v^3,
\end{array}
\end{equation}
see \eqref{eq:103}. This system has three singularities $(0,0)$,
$(-a_7/a_1,0)$ and\linebreak
$(-a_5a_8/(a_2(a_5+a_7)-a_1a_8),a_5(a_5+a_7)/(a_2(a_5+a_7)-a_1a_8))$
corresponding to singularities of $X$ on $\mathbb S^2$. The
singularity $(-a_7/a_1,0)$ corresponds to singularity
$(-a_7/\sqrt{a_7^2+a_1^2},0, -a_1/\sqrt{a_7^2+a_1^2})$ of system
\eqref{eq:142} on $\mathbb S^2$ and its eigenvalues are $
\lambda_\pm(\mu)=(\mu\pm\sqrt{\mu^2-4(a_7^2+a_5a_7)(a_1^2+a_7^2)})/(2a_1)$,
where $\mu=-(a_2a_7-a_1a_8)$. We saw that if $\mu=0$ and
$\sqrt{a_1^2+a_2^2}a_1a_5a_8/a_7<0$, $(-a_7/a_1,0)$ is a stable
weak focus. Moreover, for
$-2\sqrt{(a_7^2+a_5a_7)(a_1^2+a_7^2)}<\mu<2\sqrt{(a_7^2+a_5a_7)(a_1^2+a_7^2)}$
the singular point $(-a_7/a_1,0)$ is a strong focus and
$\displaystyle\frac{d}{d\mu}(\mbox{Re}\lambda_\pm(\mu))\mid_{\mu=0}=1$.
Hence, by the Hopf Bifurcation Theorem, we have a Hopf
bifurcation. This proves that there exist examples of homogeneous
polynomial vector fields on $\mathbb S^2$ of degree $2$ with at
least one limit cycle.

Now, consider the orthogonal linear change of variables
$(\tilde{x}, \tilde{y}, \tilde{z})=M(x,y,z)$, where $M$ is the
matrix
\[
\left(\begin{array}{ccc} \displaystyle
\frac{a_5+a_7}{\sqrt{(a_5+a_7)^2+a_8^2}} & \displaystyle
\frac{a_8(a_2(a_5+a_7)-a_1a_8)}{\sqrt{(a_5+a_7)^2+a_8^2}\sqrt{\sigma}}
& \displaystyle \frac{a_5a_8}{\sqrt{\sigma}}
\\ \displaystyle \frac{a_8}{\sqrt{(a_5+a_7)^2+a_8^2}} & \displaystyle
-\frac{(a_5+a_7)(a_2(a_5+a_7)-a_1a_8)}{\sqrt{(a_5+a_7)^2+a_8^2}\sqrt{\sigma}}
& \displaystyle -\frac{a_5(a_5+a_7)}{\sqrt{\sigma}} \\ 0 &
\displaystyle -\frac{a_5\sqrt{(a_5+a_7)^2+a_8^2}}{\sqrt{\sigma}} &
\displaystyle \frac{a_2(a_5+a_7)-a_1a_8}{\sqrt{\sigma}}
\end{array} \right).
\]
Hence, system (\ref{eq:142}) becomes
\[
\begin{array}{lcl}
\dot{\tilde{x}} & = &
\tilde{a}_1\tilde{x}\tilde{y}+\tilde{a}_2\tilde{y}^2+\tilde{a}_4\tilde{x}\tilde{z}+\tilde{a}_5\tilde{y}\tilde{z}, \\
\dot{\tilde{y}} & = & -\tilde{a}_1\tilde{x}^2-\tilde{a}_2\tilde{x}\tilde{y}+\tilde{a}_7\tilde{x}\tilde{z}, \\
\dot{\tilde{z}} & = &
-\tilde{a}_4\tilde{x}^2-(\tilde{a}_5+\tilde{a}_7)\tilde{x}\tilde{y},
\end{array}
\]
where
$\tilde{a}_1=-(a_1a_2((a_5+a_7)^2-a_8^2)+a_8(a_5a_8^2+(a_5+a_7)(a_2^2-a_1^2+a_5(a_5+a_7))))/(\sqrt{(a_5+a_7)^2+a_8^2}\sqrt{\sigma})$,
$\tilde{a}_2=(a_2(a_5+a_7)-a_1a_2)/\sqrt{(a_5+a_7)^2+a_8^2}$,
$\tilde{a}_4=-(a_8(a_2a_7-a_8a_1)-a_1(a_5^2+a_5a_7))/\sqrt{\sigma}$,
$\tilde{a}_5=a_5$, and $\tilde{a}_7=-(a_5+a_7)$. Note that in
these coordinates the singularity
$(-a_5a_8/\sqrt{\sigma},a_5(a_5+a_7)/\sqrt{\sigma},
-(a_2(a_5+a_7)-a_1a_8)/\sqrt{\sigma})$ becomes the point
$(0,0,-1)$. By Theorem \ref{the:08},
$(-a_5a_8/\sqrt{\sigma},a_5(a_5+a_7)/\sqrt{\sigma},
-(a_2(a_5+a_7)-a_1a_8)/\sqrt{\sigma})$ is a center of system
\eqref{eq:142} on $\mathbb S^2$ if $\tilde{a}_4=0$,
$\tilde{a}_4^2+\tilde{a}_5\tilde{a}_7<0$ and $
\tilde{a}_4(\tilde{a}_1^2-\tilde{a}_2^2)+\tilde{a}_1\tilde{a}_2(\tilde{a}_5+\tilde{a}_7)=0$.
We have that $\tilde{a}_4=0$ implies that
$a_8(a_2a_7-a_8a_1)-a_1(a_5^2+a_5a_7)=0$. Note that in this case
$a_8\neq 0$, otherwise
$\tilde{a}_4=-a_1(a_5^2+a_5a_7)/\sqrt{\sigma}\neq 0$. Hence,
$a_2=a_1(a_5^2+a_5a_7+a_8^2)/(a_7a_8)$ and
$\tilde{a}_4^2+\tilde{a}_5\tilde{a}_7=-(a_5^2+a_5a_7)<0$,
$\tilde{a}_4(\tilde{a}_1^2-\tilde{a}_2^2)+\tilde{a}_1\tilde{a}_2(\tilde{a}_5+\tilde{a}_7)=a_1\sqrt{a_5^2(a_8^2+(a_5+a_7)^2)(a_1^2(a_5+a_7)^2
+a_8^2(a_1^2+a_7^2))}/|a_7a_8|$. As $a_1\neq 0$,
$(-a_5a_8/\linebreak\sqrt{\sigma},a_5(a_5+a_7)/\sqrt{\sigma},
-(a_2(a_5+a_7)-a_1a_8)/\sqrt{\sigma})$ cannot be a center. Now if
$a_8(a_2a_7-a_8a_1)-a_1(a_5^2+a_5a_7)=0$, by Proposition
\ref{pro:15} it is a weak focus of system \eqref{eq:142}.

The singularity
$(-a_5a_8/(a_2(a_5+a_7)-a_1a_8),a_5(a_5+a_7)/(a_2(a_5+a_7)-a_1a_8))$
of system \eqref{eq:145} corresponds to the singularity
$(-a_5a_8/\sqrt{\sigma},a_5(a_5+a_7)/\sqrt{\sigma},
-(a_2(a_5+a_7)-a_1a_8)/\sqrt{\sigma})$ of system \eqref{eq:142} on
$\mathbb S^2$ and its eigenvalues are
$\lambda_\pm(\mu)=(\mu\pm\sqrt{\mu^2-4(a_5^2+a_5a_7)\sigma})/(a_2(a_5+a_7)-a_1a_8)$,
where $\mu=-(a_8(a_2a_7-a_8a_1)-a_1(a_5^2+a_5a_7))$. We saw that
if $\mu=0$ and $a_1/|a_7a_8|\sqrt{(a_1^2(a_5+a_7)^2
+a_8^2(a_1^2+a_7^2))}\cdot\linebreak\sqrt{a_5^2(a_8^2+(a_5+a_7)^2)}<0$,
then
$(-a_5a_8/(a_2(a_5+a_7)-a_1a_8),a_5(a_5+a_7)/(a_2(a_5+a_7)-a_1a_8))$
is a stable weak focus. Moreover, for
$-2\sqrt{(a_5^2+a_5a_7)\sigma}<\mu<2\sqrt{(a_5^2+a_5a_7)\sigma}$,
$(-a_5a_8/(a_2(a_5+a_7)-a_1a_8),a_5(a_5+a_7)/(a_2(a_5+a_7)-a_1a_8))$
is a strong focus and $\displaystyle
\frac{d}{d\mu}(\mbox{Re}\lambda_\pm(\mu))\mid_{\mu=0}=1$. Hence,
by the Hopf Bifurcation Theorem, we have a Hopf bifurcation. Note
that system \eqref{eq:142} cannot have two weak foci at the same
time. \hfill $\Box$

\section{Rotated vector field family}

In this section, we summarize the behavior of the limit cycles in
the special one--parameter family given by a rotated family of
planar vector fields. The earliest work about these families can
be found in the paper \cite{ref:24} of Duff in $1953$. Later on
Seifert \cite{ref:25}, Perko \cite{ref:26} and Chen Xiang--yan
[19, 20, 21], etc successively improved the work of Duff.

Consider the vector fields $X_\alpha(x,y)= (P(x,y,\alpha),
Q(x,y,\alpha))$ depending on the parameter $\alpha$. Suppose that
when $\alpha$ varies on an interval $(a,b)$, the singular points
of the vector fields $X_\alpha$ remain unchanged, and for any
fixed point $p=(x,y)$ and any parameters $\alpha_1<\alpha_2\in
(a,b)$, we have
\begin{equation}
\label{eq:159} \left|\begin{array}{cc}P(x,y,\alpha_1) &
Q(x,y,\alpha_1) \\ P(x,y,\alpha_2) & Q(x,y,\alpha_2)
\end{array}\right|\geq 0 \;(\mbox{or}\leq 0),
\end{equation}
where the equality cannot hold on an entire periodic orbit of
$X_\alpha$ with $\alpha=\alpha_i$, for $i=1,\;2$. Then, the family
of vector fields $X_\alpha$ is called a \index{rotated vector
field family with respect to the parameter $\alpha$}{\it (
generalized) rotated family with respect to the parameter
$\alpha$}. Here, the interval $(a,b)$ can be either bounded or
unbounded.

The geometric meaning of condition \eqref{eq:159} is the
following. At any fixed point $p=(x,y)$, the oriented area between
the vectors $(P(x,y,\alpha_1), Q(x,y,\alpha_1))$ and
$(P(x,y,\alpha_2), Q(x,y,\alpha_2))$ has the same (or opposite)
sign as sgn$(\alpha_2-\alpha_1)$. That is, at any point $p=(x,y)$,
as the parameter $\alpha$ increases, the vector $(P(x,y,\alpha),
Q(x,y,\alpha))$ can only rotate in one direction; moreover, the
angle of rotation cannot exceed $\pi$.

In the following we present four important results concerning
periodic orbits and limit cycles for rotated vector field families
$X_\alpha$.
\begin{itemize}
\item[(i)] {\it Non--intersection property:} For distinct
$\alpha_1$ and $\alpha_2$, the periodic orbits of the vector field
$X_\alpha$ with $\alpha=\alpha_1$ and of the vector field
$X_\alpha$ with $\alpha=\alpha_2$ cannot intersect. \item[(ii)]
{\it Stable and unstable property:} When the parameter $\alpha$
changes slightly and monotonically, the stable and unstable limit
cycle cannot disappear; it expands or contracts monotonically.
\item[(iii)] {\it Semiestable property:} When the parameter
$\alpha$ varies in the suitable direction, a semistable limit
cycle bifurcates into one stable and one unstable limit cycle.
When $\alpha$ varies in the opposite direction the semistable
limit cycle disappears.
\item[(iv)] {\it Ending or starting property:} When the parameter
$\alpha$ varies a family of limit cycles only can disappears or
appears either at a singular point, or in a semistable limit
cycle, or in a separatrix cycle, or at infinity (i.e. the family
becomes unbounded).
\end{itemize}

Consider system \eqref{eq:144} with $c_1\neq 0$ and $-c_5c_7<0$.
This system has six singularities, $(0,0,\pm 1)$,
$(c_7/\sqrt{c_1^2+c_7^2},0,c_1/\sqrt{c_1^2+c_7^2})$,
$(-c_7/\sqrt{c_1^2+c_7^2},0,-c_1/\sqrt{c_1^2+c_7^2})$ and two
singularities on $\mathbb S^1$, given by
$(-c_2/\sqrt{c_2^2+c_1^2},c_1/\sqrt{c_2^2+c_1^2},0)$ and
$(c_2/\linebreak\sqrt{c_2^2+c_1^2}, -c_1/\sqrt{c_2^2+c_1^2},0)$.
We saw in the proof of Theorem \ref{the:14} that the singularities
of $X$ on $\mathbb S^1$ are either nodes or strong foci on
$\mathbb S^2$.

Now we consider the planar system determined by system
\eqref{eq:144} through the central projection on the tangent plane
to $\mathbb S^2$ at the point $(0,0,-1)$, i.e.
\begin{equation}
\label{eq:160}
\begin{array}{lcl}
\dot{u} & = & \displaystyle
-c_5v+c_1uv+c_2v^2-(c_5+c_7)u^2v-\frac{c_2(c_5+c_7)}{c_1}uv^2, \\
\dot{v} & = & \displaystyle
-c_7u-\frac{c_2(c_5+c_7)}{c_1}v-c_1u^2-c_2uv-(c_5+c_7)uv^2-\frac{c_2(c_5+c_7)}{c_1}v^3,
\end{array}
\end{equation}
see \eqref{eq:103}. This system has two singularities $(0,0)$ and
$(-c_7/c_1,0)$ corresponding to the singularities of $X$ that does
not belong to $\mathbb S^1$. We saw, in Section \ref{sec:07}, that
$(0,0)$ is a saddle and, by the proof of Theorem \ref{the:14}, if
$c_2=0$ then $(-c_7/c_1,0)$ is a center; and if $c_2\neq 0$, then
it is a node or a strong focus.

\begin{proposition}
\label{pro:18} The vector field associated to system
\eqref{eq:160} does not have limit cycles.
\end{proposition}
\begin{proof}
Consider the one parameter family of vector fields $(P(u,v,c_2),
Q(u,v,c_2))$ associated to system \eqref{eq:160}. We have that
\[
\left|\begin{array}{cc}P(u,v,c_2) & Q(u,v,c_2) \\
P(u,v,\tilde{c}_2) & Q(u,v,\tilde{c}_2)
\end{array}\right|=(\tilde{c}_2-c_2)\frac{c_5(c_5+c_7)}{c_1}(1+u^2+v^2)v^2.
\]
Therefore, system \eqref{eq:160} determines a rotated vector field
family with respect to the parameter $c_2$. Now, note that for
$c_2=0$, $(-c_7/c_1,0)$ is a center of this family. Then, by the
non--intersection property (i) and the ending or starting property
(iv), system \eqref{eq:160} does not have limit cycles.
\end{proof}

We conclude by Proposition \ref{pro:18} that if the vector field
$X$ on $\mathbb S^2$ associated to system \eqref{eq:144} with
$c_1\neq 0$ and $-c_5c_7<0$ has a limit cycle, then this limit
cycle surrounds one of the singularities that belongs to $\mathbb
S^1$. These singularities always are nodes or strong foci. We
believe that the following conjecture holds.

\medskip \noindent{\bf Conjecture} The vector field associated to
system \eqref{eq:144} with $c_1\neq 0$ and $-c_5c_7<0$ does not
have limit cycles on $\mathbb S^2$.


\addcontentsline{toc}{chapter}{Bibliografia}
\begin{thebibliography}{6}

\bibitem{ref:7} {\sc A.A. Andronov, E.A. Leontovich, I.I. Gordon and A.G. Maier},
Qualitative theory of second-order dynamic systems, John Wiley \&
Sons, $1973$.

\bibitem{ref:2} {\sc M.I. Camacho},
Geometric properties of homogeneous vector fields of degree two in
$\mathbb R^{3}$, {\it Trans. Amer. Math. Soc.} {\bf 268} (1981),
79--101.

\bibitem{ref:21} {\sc J. Chavarriga and J. Llibre},
Invariant algebraic curves and rational first integrals for planar
polynomial vector fields, {\it J. Differential Equations} {\bf
169} (2001), 1--16.

\bibitem{ref:8} {\sc G. Darboux}, M\'emoire sur les \'equations
diff\'erentielles alg\'ebrique du premier ordre et du premier
degr\'e (M\'elanges), {\it Bull. Sci. Math.} 2\'eme s\'erie, {\bf
2} (1878), 60--96; 123--144; 151--200.

\bibitem{ref:24} {\sc G. F. D. Duff}, Limit cycles and rotated vector
fields, {\it Ann. of Math.} {\bf 57} (1953), 15--31.

\bibitem{ref:20} {\sc T. Fukushima and T. Yamamoto},
A quasi--regularization of the Euler problem by partitioned
multistep methods, {\it Proceedings of the 36th Symposium on
Celestial Mechanics}, March 1--3, 2004 at Seiun--So, Hakone,
Kanagawa, Japan.

\bibitem{ref:13} {\sc W. Fulton}, Algebraic curves, New York: W. A. Benjamin Inc., $1969$.

\bibitem{ref:19} {\sc H. Giacomini, J. Llibre and M. Viano},
On the nonexistence, existence and uniqueness of Limit cycles,
{\it Nonlinearity} {\bf 9} (1996), 501--516.

\bibitem{ref:3} {\sc C. Gutierrez and J. Llibre},
Darbouxian integrability for polynomial vector fields on the
$2$-dimensional sphere, {\it Extracta Mathematicae} {\bf 17}
(2002), 289--301.

\bibitem{ref:22} {\sc C. Lansun, Z. Xinan and L. Zhaojun},
Global topological properties of homogeneous vector fields in
$\mathbb R^3$, {\it Chin. Ann. of Math.}, Ser. 20 B, {\bf 2}
(1999), 185--194.

\bibitem{ref:15} {\sc J. Llibre and C. Pessoa},
Homogeneous polynomial vector fields of degree $2$ on the
$2$--dimensional sphere, preprint, 2006.

\bibitem{ref:17} {\sc J. Llibre and C. Pessoa},
Invariant circles for homogeneous polynomial vector fields on the
$2$--dimensional sphere, to appear in {\it Rend. Cir. Mat.
Palermo}, 2006.

\bibitem{MHNM} {\sc M. Makhaniok, J. Hesser, S. Noehte andR.
M\"{a}nner}, {\it Limit cycles of a planar vector field}, Acta
Applicandae Mathematicae {\bf 48} (1997), 12--32.

\bibitem{ref:14} {\sc J. W. Milnor},
Topology from the differentiable viewpoint, Princeton Landmarks in
Math., Princeton Universit Press, New Jersey 1997.

\bibitem{ref:23} {\sc J. E. Marsden and M. McCracken},
The Hopf bifurcation and its applications, Applied Mathematical
Sciences, Vol. {\bf 19}, Springer, New York 1976.

\bibitem{ref:26} {\sc L. Perko}, Rotated vector fields and the global behavior of limit
cycles for a class of quadratic systems in the plane, {\it J.
Differential Equations} {\bf 18} (1975), 63--86.

\bibitem{ref:18} {\sc C. Pessoa},
Homogeneous polynomial vector fields on the $2$--dimensional
sphere, {\it Ph. D. thesis, Universitat Aut\`onoma de Barcelona},
2006.

\bibitem{ref:25} {\sc G. Seifert}, A rotated vector aproach to the problem of stability of
solutions of pendulum--type equations, in Contributions to the
theory of nonlinear oscilations, {\it Annals of Math. Studies}
{\bf 3} (1956), 1--16.

\bibitem{ref:27} {\sc C. Xiang--Yan}, Applications of the theory of rotated vector fields
I, {\it Nanjing Daxue Xuebao (Math.)} {\bf 1} (1963), 19--25.

\bibitem{ref:28} {\sc C. Xiang--Yan}, Applications of the theory of rotated vector fields
II, {\it Nanjing Daxue Xuebao (Math.)} {\bf 2} (1963), 43--50.

\bibitem{ref:29} {\sc C. Xiang--Yan}, On generalized rotated vector
fields, {\it Nanjing Daxue Xuebao (Nat. Sci.)} {\bf 1} (1975),
100--108.

\end{thebibliography}
\end{document}